\documentclass{article}
\usepackage{graphicx} 
\usepackage{amssymb,latexsym}
\usepackage{amscd}
\usepackage{euscript}
\usepackage{amsmath}
\usepackage{color}
\usepackage{tikz-cd}
\usepackage[all,cmtip]{xy}
\usepackage{comment}
\usepackage{orcidlink}
\usepackage{hyperref}

\newtheorem{definition}{\bf Definition}[section]
\newtheorem{remark}{\bf Remark}[section]
\newtheorem{lemma}{\bf Lemma}[section]
\newtheorem{theorem}{\bf Theorem}[section]

\newtheorem{proposition}{\bf Proposition}[section]
\newtheorem{example}{\bf Example}[section]

\newenvironment{proof}
{\par\noindent\textbf{Proof. }\rm}
{\hfill$\boldsymbol{\square}$\par\vspace{0.3cm}}

\title{From Multirelations to Meet-Relations:
A Relational Duality for Semilattices with Adjunctions}
 \author{B. Gimenez, W. Zuluaga}
\date{}

\begin{document}

\maketitle

\begin{abstract}
\noindent
We develop a relational duality for semilattices with adjunctions (SLatas) based on binary meet-relations. First, we introduce the category of MoS-spaces and establish a dual equivalence with modal semilattices. Then, by means of A-relations, we define the category $\mathsf{RelSp}$ and prove a dual equivalence between $\mathsf{SLata}$ and $\mathsf{RelSp}$. To compare this framework with the multirelational semantics previously developed for SLatas, we introduce the notion of normal mS-space and show that, under this condition, the multirelational structure can be canonically recovered from a meet-relation, and conversely. As a consequence, we prove that the categories $\mathsf{RelSp}$ and $\mathsf{SLataSp}$ are isomorphic.
\end{abstract}

\section{Introduction}

Duality theory provides a fundamental bridge between algebraic structures and their associated semantic models. By establishing contravariant correspondences between algebraic categories and categories of structured spaces, dualities allow algebraic phenomena to be studied through topological, order-theoretic, or relational methods.  In the specific setting of semilattices, this perspective has proven particularly fruitful: the duality developed in \cite{CG} between semilattices and S-spaces provides a flexible framework for representing algebraic expansions, while subsequent work extends this approach to monotone semilattices and semilattices with adjunctions through multirelational structures \cite{P W,GPZ}.

Despite these developments, the use of multirelations that is, subsets of
$
X\times \mathcal{P}(X)
$
instead of ordinary binary relations
$
X\times X
$, comes at a conceptual and technical cost. While multirelations faithfully encode monotone operators and adjoint pairs at the dual level, they also move the semantics to a genuinely higher-order setting, obscuring the first-order character that is typical of classical relational semantics and making the associated dual structures more difficult to manipulate. This naturally leads to the following question: can the multirelational component in dualities for expanded semilattices be replaced by purely binary relational data without losing the underlying categorical information?
\medskip

The present paper answers this question affirmatively for the case of semilattices with adjunctions (SLatas), namely algebras
$
\langle A,i,d\rangle
$
where the unary operations $i$ and $d$ form an adjoint pair $i\dashv d$. Our main contribution is to show that, under a suitable normality condition on the associated dual spaces, the multirelational semantics introduced in \cite{GPZ} can be canonically reconstructed from binary meet-relations. More precisely, we define a category $\mathsf{RelSp}$ of relational S-spaces and establish a dual equivalence between $\mathsf{SLata}$ and $\mathsf{RelSp}$. Furthermore, we prove that $\mathsf{RelSp}$ is not merely equivalent to the multirelational category $\mathsf{SLataSp}$ introduced in \cite{GPZ}, but actually isomorphic to it.

The categorical relationships developed throughout the paper are summarized in the following diagram:
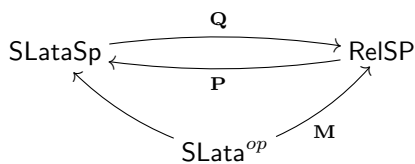
\[
\xymatrix{
\mathsf{SLataSp}
\ar@/^/[rr]^{\mathbf{Q}}
&&
\mathsf{RelSp}
\ar@/^/[ll]^{\mathbf{P}}
\\
&
\mathsf{SLata}^{op}
\ar@/^/[ul]
\ar@/_/[ur]_{\mathbf{M}}
&
}
\]

The lower part of the diagram displays two dual representations of semilattices with adjunctions: on the left, the multirelational semantics developed in \cite{GPZ}, and on the right, the relational semantics based on meet-relations introduced in the present paper. The upper part shows that the corresponding categories of spaces are, in fact, isomorphic. The functor \(\mathbf{M}\) is defined in Section~\ref{Relational Duality for SLata}, whereas the functors \(\mathbf{P}\) and \(\mathbf{Q}\), which establish the isomorphism between \(\mathsf{RelSp}\) and \(\mathsf{SLataSp}\), are defined in Section~\ref{the isomorphism of categories}.

Our strategy proceeds in two stages. First, we develop a relational semantics for modal semilattices, namely semilattices equipped with a meet-preserving operator. To this end, we introduce the category $\mathsf{MoSsp}$ of MoS-spaces and establish its dual equivalence with the algebraic category $\mathsf{MoSL}$. This provides the fundamental relational machinery---meet-relations, compatible morphisms, and the associated functorial constructions---that will later be extended to the adjoint setting.

In the second stage, we enrich this framework in order to represent adjoint pairs on semilattices. For this purpose, we introduce A-relations (adjoint-relations) and adjoint-preserving morphisms, obtaining the category $\mathsf{RelSp}$ and the corresponding duality for SLatas. We then introduce the notion of normal $mS$-space and prove that every relevant multirelational structure can be canonically recovered from an associated meet-relation. This ultimately yields the main result of the paper: the categories $\mathsf{SLataSp}$ and $\mathsf{RelSp}$ are isomorphic.

The paper is organized as follows. Section \ref{Preliminaries} reviews the necessary background concerning semilattices, S-spaces, monotone semilattices, and SLatas. In particular, we recall the topological duality for semilattices from \cite{CG} together with its multirelational extensions developed in \cite{P W,GPZ}. Section \ref{A Relational Duality for SLata} develops the relational duality for modal semilattices and extends it to semilattices with adjunctions via A-relations, culminating in the dual equivalence between $\mathsf{SLata}$ and $\mathsf{RelSp}$. Section \ref{sec: The isomorphism} contains the main technical contribution of the paper. There we introduce normal $mS$-spaces and prove the isomorphism between $\mathsf{SLataSp}$ and $\mathsf{RelSp}$. Finally, we conclude with a discussion on the algebraic and semantic meaning of the constructions introduced throughout the paper.

\section{Preliminaries}\label{Preliminaries}
The aim of this section is twofold. On the one hand, we fix the notation and terminology used throughout the paper. On the other hand, we review the dual frameworks underlying our approach, namely the topological duality for semilattices and its multirelational extensions to monotone semilattices and SLatas.
\medskip

We begin with some basic notation. Let $f:A\to B$ be a function, and let $U\subseteq A$ and $V\subseteq B$. We denote by $f[U]$ and $f^{-1}[V]$ the direct and inverse images of $U$ and $V$ under $f$, respectively. Let $\langle X, \leq\rangle$ be a partially ordered set. For every $x \in X$, we write $\uparrow x$ for the set $\{y \in X : x \leq y\}$. A subset $U \subseteq X$ is called an \emph{upset} if, whenever $x \leq y$ and $x \in U$, then $y \in U$. The set of upsets of $X$ is denoted by $\mathrm{Up}(X)$. The notion of \emph{downset} is defined dually. A non-empty subset $Z \subseteq X$ is called \emph{dually directed} if for all $x,y \in Z$ there exists $z \in Z$ such that $z \leq x$ and $z \leq y$. A \emph{meet-semilattice with greatest element}, or simply a \emph{semilattice}, is an algebra $\mathbf{A} = \langle A, \wedge, 1 \rangle$ such that $\wedge$ is idempotent, commutative, and associative, and $a \wedge 1 = a$ for all $a \in A$. The associated order is given by $a \leq b$ if and only if $a = a \wedge b$. It is well known that $\langle \mathrm{Up}(X), \cap, X\rangle$ is a semilattice.

\subsection{Topological duality for semilattices}

Let $\mathbf{A}$ be a semilattice. A subset $F \subseteq A$ is a \emph{filter} if it is an upset, $1 \in F$, and $a,b \in F$ implies $a \wedge b \in F$. We denote the set of filters of $\mathbf{A}$ by $\mathrm{Fi}(\mathbf{A})$. A proper filter $F$ is \emph{irreducible} if, whenever $F = F_1 \cap F_2$ for filters $F_1, F_2$, then $F = F_1$ or $F = F_2$. We write $\mathcal{X}(\mathbf{A})$ for the set of irreducible filters of $\mathbf{A}$. A non-empty subset $I \subseteq A$ is called an \emph{order-ideal} of $\mathbf{A}$ if it is a directed downset. That is, $I$ is a downset such that every pair of elements $a,b\in I$ has an upper bound in $I$. We denote by $\mathrm{Id}(\mathbf{A})$ the set of all order-ideals of $\mathbf{A}$.

Consider the poset $(\mathcal{X}(\mathbf{A}), \subseteq)$ and the map $\beta_{\mathbf{A}} \colon A \to \mathrm{Up}(\mathcal{X}(\mathbf{A}))$ given by
\[
\beta_{\mathbf{A}}(a) = \{P \in \mathcal{X}(\mathbf{A}) : a \in P\}.
\]
It was shown in \cite{CG} that $\mathbf{A}$ is isomorphic to the subalgebra of $\langle\mathrm{Up}(\mathcal{X}(\mathbf{A})), \cap, \mathcal{X}(\mathbf{A})\rangle$ with universe $\beta_{\mathbf{A}}[A]$. Throughout the paper we omit the subscript when no confusion arises.

\medskip

By a \emph{topological space} we mean a pair $\langle X,\mathcal{K}\rangle$, where $\mathcal{K} \subseteq \mathcal{P}(X)$ is a subbase such that $X = \bigcup \mathcal{K}$. We set
\[
S(X) = \{U^c : U \in \mathcal{K}\}.
\]
We denote by $\mathcal{C}_{\mathcal{K}}(X)$ the closure system generated by $S(X)$, that is,
\[
\mathcal{C}_{\mathcal{K}}(X) = \left\{ \bigcap \mathcal{A} : \mathcal{A} \subseteq S(X) \right\}.
\] 

\

Given $Y \subseteq X$, a family $\mathcal{J} \subseteq S(X)$ is called a \emph{$Y$-family} if for all $A,B \in \mathcal{J}$ there exist $H,C \in S(X)$ such that $Y \subseteq H$, $C \in \mathcal{J}$, $A \cap H \subseteq C$, and $B \cap H \subseteq C$.

\

An \emph{$S$-space} \cite{CG} is a topological space $\langle X,\mathcal{K}\rangle$ satisfying:
\begin{itemize}
    \item[(S1)] $\langle X,\mathcal{K}\rangle$ is $T_0$ and $X = \bigcup \mathcal{K}$;
    \item[(S2)] $\mathcal{K}$ is a subbase of compact open sets, closed under finite unions, and $\emptyset \in \mathcal{K}$;
    \item[(S3)] if $x \in U \cap V$ with $U,V \in \mathcal{K}$, then there exist $W,D \in \mathcal{K}$ such that $x \notin W$, $x \in D$, and $D \subseteq (U \cap V)\cup W$;
   \item[(S4)] If $Y \in \mathcal{C}_{\mathcal{K}}(X)$ and $\mathcal{J} \subseteq S(X)$ is a $Y$-family such that $Y \cap A^c \neq \emptyset$ for all $A \in \mathcal{J}$, then 
\[
Y \cap \bigcap \{A^c : A \in \mathcal{J}\} \neq \emptyset.
\]
\end{itemize}

Let $\langle X, \mathcal{K}\rangle$ be an S-space and let $Y \subseteq X$. We write $\text{cl}_{\mathcal{K}}(Y)$ for the topological closure of $Y$. That is, 
\[\text{cl}_{\mathcal{K}}(Y)=\bigcap\{U \in S(X) \colon Y\subseteq U\}.\]
In particular, if $Y = \{y\}$, then we simply denote $\text{cl}_\mathcal{K}(\{y\})$ by $\text{cl}_\mathcal{K}(y)$. Observe that by (S1), $X$ can be endowed with a partial order defined by $x \sqsubseteq y$ if and only if $x \in \text{cl}_\mathcal{K}(y)$. Such an order is called the \emph{specialization order} of $X$. We write $\sqsupseteq$ for the dual of the specialization order. Observe that, relative to the poset $\langle X, \sqsupseteq\rangle$, for every $x\in X$ we have that $\uparrow x=\text{cl}_{\mathcal{K}}(x)$. If $\mathbf{A}$ is a semilattice, then both $\sqsupseteq$ and $\subseteq$, coincide on $\mathcal{X}(\mathbf{A})$. 
\\

If $\langle X,\mathcal{K}\rangle$ is an $S$-space, then $\mathbf{S}(X) = \langle S(X), \cap, X\rangle$ is a semilattice. Conversely, if $\mathbf{A}$ is a semilattice, then $\langle\mathcal{X}(\mathbf{A}), \mathcal{K}_{\mathbf{A}}\rangle$ is an $S$-space, where $\mathcal{K}_{\mathbf{A}} = \{\beta(a)^c : a \in A\}$, and $\beta \colon A \to S(\mathcal{X}(\mathbf{A}))$ is an isomorphism. Moreover, for every $S$-space $(X,\mathcal{K})$, the map $H_X \colon X \to \mathcal{X}(\mathbf{S}(X))$ given by
\begin{equation}\label{eq: H_X}
H_X(x) = \{A \in S(X) : x \in A\}    
\end{equation}
is a homeomorphism (see \cite{CG}).

\medskip

Let $\mathsf{Rel}$ be the category of sets and binary relations. There is a faithful functor $\Box \colon \mathsf{Rel}^{op} \to \mathsf{Set}$ given by $\Box(X)=\mathcal{P}(X)$ and, for $T \subseteq X \times Y$,
\[
\Box_T(U) = \{x \in X : T(x) \subseteq U\}.
\]

A \emph{meet-relation} between $S$-spaces $\langle X_1,\mathcal{K}_1\rangle$ and $\langle X_2,\mathcal{K}_2\rangle$ is a relation $T \subseteq X_1 \times X_2$ such that:
\begin{enumerate}
    \item $\Box_T(U) \in S(X_1)$ for all $U \in S(X_2)$;
    \item $T(x) = \displaystyle \bigcap \{U \in S(X_2) : T(x) \subseteq U\}$ for all $x \in X_1$.
\end{enumerate}

Given $S$-spaces $\langle X_i,\mathcal{K}_i\rangle$ for $i=1,2,3$, and meet-relations $R \subseteq X_1 \times X_2$ and $T \subseteq X_2 \times X_3$, their composition is defined by
\begin{equation}\label{Definicion composition meet-relations}
T \ast R := \{(x,z) \in X_1 \times X_3 : \forall U \in S(X_3)\, ((T \circ R)(x) \subseteq U \Rightarrow z \in U)\}.
\end{equation}
This operation is associative. If $\sqsupseteq$ denotes the dual specialization order, then $T \ast \sqsupseteq = T$ and $\sqsupseteq \ast R = R$. Hence, $S$-spaces with meet-relations form a category $\mathsf{Sspa}$. Let $\mathsf{S}$ be the category of semilattices and homomorphisms. We have:

\begin{theorem}[\cite{CG}]\label{theo: dualidad semireticulos}
The categories $\mathsf{S}$ and $\mathsf{Sspa}$ are dually equivalent.
\end{theorem}

\subsection{Multi-relational duality for monotone semilattices}

Let $\mathbf{A}$ be a semilattice. A map $m \colon A \to A$ is said to be \emph{monotone} if it preserves the underlying order of $\mathbf{A}$. Equivalently, $m$ is monotone if, for all $a,b \in A$,
\[
m(a \wedge b) \leq m(a) \wedge m(b).
\]
A \emph{monotone semilattice} is a pair $\langle\mathbf{A}, m\rangle$ where $\mathbf{A}$ is a semilattice and $m$ is a monotone operation on $\mathbf{A}$. It is straightforward to verify that the class of monotone semilattices forms a variety.

\medskip

Given two monotone semilattices $\langle\mathbf{A}, m\rangle$ and $\langle\mathbf{B}, n\rangle$, a semilattice homomorphism $h \colon \mathbf{A} \to \mathbf{B}$ is called a \emph{monotone homomorphism} if
\[
h(m(a)) = n(h(a)) \quad \text{for all } a \in A.
\]
Morphisms compose via the standard composition of functions, and the identity morphisms are precisely the identity maps. We write $\mathsf{mMS}$ for both the category of monotone semilattices with monotone homomorphisms and the corresponding variety.

\medskip

We now turn to the relational structures used to represent monotone operations on the dual side. Let $\langle X,\mathcal{K} \rangle$ be an $S$-space. A subset $Z \subseteq X$ is called a \emph{subbasic saturated subset} if there exists a dually directed family $\mathcal{L} \subseteq \mathcal{K}$ (w.r.t. $\subseteq$) such that
\[
Z = \bigcap \mathcal{L}.
\]
We denote by $\mathcal{Z}(X)$ the set of subbasic saturated subsets of $X$.

\begin{remark}\label{rem: Remark1 CMZ}
    Let $Z\in \mathcal{Z}(X)$, with $Z=\bigcap\mathcal{L}$ for some directed family $\mathcal{L}\subseteq \mathcal{K}$. If $Y\in \mathcal{C}_{\mathcal{K}}(X)$ satisfies $Y\cap Z=\emptyset$, then by Remark 1 of \cite{P W} and condition (S4), there exists $U\in \mathcal{L}$ such that $Y\cap U=\emptyset$.
\end{remark}

A \emph{multirelation on a set $X$} is a subset of the Cartesian product $X \times \mathcal{P}(X)$, that is, a set of pairs $(x,Y)$ with $x \in X$ and $Y \subseteq X$.

\medskip
\begin{definition}\label{def: mS-space}   
An mS-space is a structure $\langle X, \mathcal{K}, R \rangle$, where $\langle X, \mathcal{K} \rangle$ is an $S$-space and $R \subseteq X \times \mathcal{Z}(X)$ is a multirelation such that:
\begin{itemize}
    \item[(m1)] $m_R(U) = \{ x \in X : R(x)\subseteq L_{U} \} \in S(X)$ for all $U \in S(X)$,
    \item[(m2)] $R(x) = \bigcap\{ L_U : U \in S(X) \text{ and } x \in m_R(U) \}$ for all $x \in X$,
\end{itemize}
where, for every $U \in S(X)$, $L_{U} = \{ Z \in \mathcal{Z}(X) \colon Z \cap U \neq \emptyset \}$.
\end{definition}
\medskip

Let $\langle X_1, \mathcal{K}_1, R_1 \rangle$ and $\langle X_2, \mathcal{K}_2, R_2 \rangle$ be $mS$-spaces. A meet-relation $T \subseteq X_1 \times X_2$ is called a \emph{monotone meet-relation} if the following diagram commutes:
\[
\xymatrix{
S(X_2) \ar[r]^{\Box_T} \ar[d]_{m_{R_2}} & S(X_1) \ar[d]^{m_{R_1}} \\
S(X_2) \ar[r]_{\Box_T} & S(X_1)
}
\]

Moreover, the composition $\ast$ defined in \eqref{Definicion composition meet-relations} is closed on monotone meet-relations (see \cite{P W}). Since $\ast$ is associative and the dual specialization order yields identity morphisms, mS-spaces together with monotone meet-relations form a category, denoted by $\mathsf{mSp}$. Hence:

\begin{theorem}[\cite{P W}]\label{duality Msp mMS}
    The categories $\mathsf{mSp}$ and $\mathsf{mMS}$ are dually equivalent.
\end{theorem}

\subsection{Multi-relational duality for SLatas}

Let $\langle P,\leq \rangle$ and $\langle Q,\leq\rangle$ be posets, and let $f \colon P \to Q$ and $g \colon Q \to P$ be maps. The pair $(f,g)$ is said to be an \emph{adjoint pair} if, for all $p \in P$ and $q \in Q$,
\[
f(p) \leq q \quad \text{if and only if} \quad p \leq g(q).
\]
In this case, $f$ is the \emph{left adjoint} and $g$ the \emph{right adjoint}, and we write $f \dashv g$. In particular, both maps are monotone. When $P=Q$, the notation $f \dashv g$ indicates \emph{an adjunction on $P$}.

\medskip
\begin{definition}

Let $\mathbf{A}$ be a semilattice and let $i,d \colon A \to A$ be unary operations. A semilattice with an adjunction (briefly, an \emph{SLata}) is an algebra $\langle \mathbf{A}, i, d \rangle$ such that $i \dashv d$.
\end{definition}
Given two SLata $\langle \mathbf{A}_1, i_1, d_1\rangle$ and $\langle \mathbf{A}_2, i_2, d_2\rangle$, a map $h \colon A_1 \to A_2$ is an \emph{SLata homomorphism} if it is a semilattice homomorphism preserving the adjoint operations, that is,
\[
h(i_1(a)) = i_2(h(a)) \quad \text{and} \quad h(d_1(a)) = d_2(h(a))
\]
for all $a \in A_1$.

We denote by $\mathsf{SLata}$ both the variety of SLata and the category whose objects are SLata and whose morphisms are SLata morphisms. Morphisms compose via the usual composition of functions, and identities are given by identity maps.

\medskip
\begin{definition}
An SLata-space is a structure $\langle X,\mathcal{K},I,E \rangle$ such that:
\begin{itemize}
    \item[(A1)] $\langle X,\mathcal{K},I\rangle$ and $\langle X,\mathcal{K},E\rangle$ are mS-spaces;
    \item[(A2)] For every $U \in S(X)$, if $x \in U$ then, for every $Z \in E(x)$, there exists $w \in Z$ such that $I(w) \subseteq L_U$;
    \item[(A3)] For every $U \in S(X)$, if for every $Z \in I(x)$ there exists $y \in Z$ such that $E(y) \subseteq L_U$, then $x \in U$.
\end{itemize}
\end{definition}
\medskip

Let $\langle X_1,\mathcal{K}_1,I_1,E_1\rangle$ and $\langle X_2,\mathcal{K}_2,I_2,E_2\rangle$ be SLata-spaces. A meet-relation $T \subseteq X_1 \times X_2$ is called an \emph{SLata-relation} if it is monotone as a relation between the $mS$-spaces $\langle X_1,\mathcal{K}_1,I_1\rangle$ and $\langle X_2,\mathcal{K}_2,I_2\rangle$, and also between $\langle X_1,\mathcal{K}_1,E_1\rangle$ and $\langle X_2,\mathcal{K}_2,E_2\rangle$.

\medskip

Given SLata-spaces $\langle X_j, \mathcal{K}_j, I_j, E_j\rangle$ for $j=1,2,3$, and SLata-relations $H \subseteq X_1 \times X_2$ and $T \subseteq X_2 \times X_3$, their composition is defined by $T \circ H := T \ast H$. This operation is well defined and associative. Moreover, for every SLata-space $\langle X, \mathcal{K}, I, E \rangle$, the dual of the specialization order $\sqsupseteq \subseteq X \times X$ is an SLata-relation and acts as the identity.

Consequently, SLata-spaces together with SLata-relations form a category, denoted by $\mathsf{SLataSp}$. Thus:

\begin{theorem}[\cite{GPZ}]\label{duality SLata}
    The categories $\mathsf{SLataSp}$ and $\mathsf{SLata}$ are dually equivalent.
\end{theorem}

\section{A Relational Duality for $\mathsf{SLata}$}\label{A Relational Duality for SLata}

\subsection{Relational Duality for SLata}

The main goal of this section is to obtain a relational duality for SLatas, in contrast with the multirelational duality developed in \cite{GPZ}. To achieve this, we develop a purely relational framework in which both adjoint operations are represented at the level of relations on S-spaces. We begin from the duality for modal semilattices, where only the right adjoint operation is taken into account. This modal case plays a pivotal role in the construction, as it provides the underlying relational semantics in terms of meet-relations and establishes the correspondence between semilattice structure and relational structure.

In order to pass from modal semilattices to full SLatas, it is necessary to incorporate the additional structure given by the left adjoint operator. This requires refining both the relational semantics and the notion of morphism, since meet-relations alone are not sufficient to encode adjoint pairs.

The guiding idea is to isolate a class of relations that simultaneously capture the right adjoint and guarantee the existence of a compatible left adjoint at the relational level. This leads to the notion of A-relations and to the category of RelS-spaces, which will serve as the relational counterpart of SLatas.

This section builds on the duality of Celani and González \cite{CG} between 
semilattices and S-spaces with meet-relations. The S-spaces and the 
composition $\ast$ of meet-relations are used here as in \cite{CG}. There, 
a meet-relation only has to satisfy conditions (1)--(2) in order to represent 
a semilattice homomorphism, whereas an A-relation is a meet-relation carrying 
one additional condition --- namely, that $\bigcap F_U^T \in S(X)$ for every 
$U \in S(X)$ --- which is precisely what is needed to guarantee the existence 
of a left adjoint for $\Box_T$. We use the term ``meet-relation'' as in 
\cite{CG} on the S-space side only; the comparison with F-continuous functions 
between HMS-spaces developed in \cite[Section~5]{CG} is not needed here.

The rest of the subsection develops this construction and shows that it yields a dual equivalence between SLatas and their relational semantics.

\subsection{Relational Duality for Modal Semilattices}

In this subsection we develop a relational duality for modal semilattices. This construction plays a central role in the sequel, since the duality for SLatas obtained in the next section may be understood as an extension of the modal case from a single operator to a pair of adjoint operators. The results established here provide the basic relational machinery that will be repeatedly used later on.

\begin{definition}
    A modal semilattice is an algebra $\langle \mathbf{A}, d \rangle$, where $\mathbf{A}$ is a semilattice and $d: A \to A$ is an operator satisfying the following conditions:  
\begin{enumerate}
    \item $d(1) = 1$;  
    \item $d(a \wedge b) = d(a) \wedge d(b)$ for all $a, b \in A$.  
\end{enumerate}
\end{definition}

We denote by $\mathsf{MoSL}$ the category whose objects are modal semilattices. 
Given two modal semilattices $\langle \mathbf{A}, d \rangle$ and $\langle \mathbf{B}, n \rangle$, 
a semilattice homomorphism $h \colon \mathbf{A} \to \mathbf{B}$ is called a \emph{modal homomorphism} if
\[
  h(d(a)) = n(h(a))
\]
for all $a \in A$. 
The identity morphism on a modal semilattice is the usual identity function, 
and the composition of morphisms in $\mathsf{MoSL}$ is given by the usual composition of functions.
\medskip

In order to formulate a duality for modal semilattices, we enrich $S$-spaces with a meet-relation representing the modal operator at the relational level.

\begin{definition}
    We say that $\langle X,\mathcal{K},T \rangle$ is a Modal S-space (MoS-space, for short) if $\langle X,\mathcal{K}\rangle$ is an S-space 
    and $T\subseteq X \times X$ is a meet-relation.
\end{definition}

The next result shows that meet-relations naturally induce modal operators on the semilattice of subbasic closed sets. Hence, MoS-spaces provide relational representations of modal semilattices.

\begin{proposition}\label{S(X) es modal}
    Let $\langle X, \mathcal{K}, T \rangle$ be a MoS-space. Then,  $\langle \mathbf{S}(X),\Box_T\rangle$ is a modal semilattice.
\end{proposition}
\begin{proof}
    The proof follows immediately from the Proposition 3.19 of \cite{CG}.
\end{proof}

\begin{remark}\label{definicion D_m}
Let $\mathbf{A}$ and $\mathbf{B}$ be semilattices, and let $h:A\to B$ be a semilattice homomorphism. By Proposition~3.20 of \cite{CG}, there is an associated meet-relation $N_h \subseteq \mathcal{X}(\mathbf{B}) \times \mathcal{X}(\mathbf{A})$ defined by
\[
(P,Q) \in N_h \quad \Longleftrightarrow \quad h^{-1}[P] \subseteq Q.
\]

In particular, if $\langle \mathbf{A}, d \rangle$ is a modal semilattice, then $d \colon A \to A$ is a semilattice homomorphism. Hence, we may consider the associated meet-relation $N_d \subseteq \mathcal{X}(\mathbf{A}) \times \mathcal{X}(\mathbf{A})$. Consequently, the structure
\[
\langle \mathcal{X}(\mathbf{A}), \mathcal{K}_{\mathbf{A}}, N_d \rangle
\]
is a MoS-space.
\end{remark}

We now introduce the appropriate notion of morphism between MoS-spaces. Given MoS-spaces $\langle X,\mathcal{K}_X,T_1 \rangle$ and $\langle Y,\mathcal{K}_Y,T_2 \rangle$, a meet-relation $M \subseteq X \times Y$ is said to be \emph{compatible} if
\[
M * T_1 \;=\; T_2 * M.
\]

Moreover, Proposition~3.17 of \cite{CG} ensures that for every $S$-space $\langle X, \mathcal{K}\rangle$, the dual specialization order $\sqsupseteq_X$ is itself a meet-relation. Hence, if $\langle X, \mathcal{K}, T\rangle$ is an MoS-space, then
\[
\sqsupseteq_X * T \;=\; T * \sqsupseteq_X.
\]

\begin{proposition}
 Let $\langle X_{i},\mathcal{K}_{i},T_{i}\rangle$, with $i=1,2,3$, be MoS-spaces, and let $M_{1} \subseteq X_{1} \times X_{2}$ and $M_{2} \subseteq X_{2} \times X_{3}$ be compatible meet-relations. Then their composition $M_{2} * M_{1}$ is a compatible meet-relation, i.e.,
\[
(M_{2} * M_{1}) * T_{1} \;=\; T_{3} * (M_{2} * M_{1}).
\]
\end{proposition}

\begin{proof}
This follows directly from the associativity of $*$ and the fact that $M_1$ and $M_2$ are compatible meet-relations.
\end{proof}

The previous results show that compatible meet-relations are stable under composition, while the dual specialization order acts as the identity morphism. Consequently, MoS-spaces together with compatible meet-relations form a category, denoted by $\mathsf{MoSsp}$, where composition is given by \eqref{Definicion composition meet-relations}.

\begin{proposition}\label{S(X(A)) es modal}
    Let $\langle \mathbf{A}, d \rangle$ be a modal semilattice. Then, $\langle \mathbf{S}(\mathcal{X}(\mathbf{A})), \Box_{N_d} \rangle$ is a modal semilattice.
\end{proposition}

\begin{proof}
    First, note that by Proposition 3.19 of \cite{CG} we have that $\Box_{N_d}\colon \mathbf{S}(\mathcal{X}(\mathbf{A})) \rightarrow \mathbf{S}(\mathcal{X}(\mathbf{A})) $ is a homomorphism of semilattices. Moreover, consider $a, b \in A$ so, $\Box_{N_d}(\beta(a)\cap \beta (b))=\Box_{N_d}(\beta(a \wedge b))$. Then, from Proposition 3.20 of \cite{CG} we have $\Box_{N_d}(\beta(a \wedge b))=\beta (d(a \wedge b))= \beta (d(a)) \cap \beta(d(b))$. Thus, we conclude $\Box_{N_d}(\beta(a)\cap \beta (b))= \Box_{N_d}(\beta(a))\cap \Box_{N_d}(\beta(b))$. The proof that $\Box_{N_d}$ preserves $\beta(1)$ is analogous.
\end{proof}
\begin{proposition}\label{R_h preserva adj derecha}
Let $\langle \mathbf{A},d\rangle$ and $\langle \mathbf{B},n\rangle$ be modal semilattices, and let $h\colon \mathbf{A}\to \mathbf{B}$ be a modal homomorphism. 
Then, the associated meet-relation 
$N_h\subseteq \mathcal{X}(\mathbf{B})\times \mathcal{X}(\mathbf{A})$ satisfies
\[
N_h*N_n=N_d*N_h .
\]
\end{proposition}
\begin{proof}
    By Lemma~3.21 of \cite{CG}, it suffices to prove that $\Box_{N_h * N_n} = \Box_{N_d * N_h}$. Let $a \in A$. By Proposition~3.22 of \cite{CG}, we have
\[
\Box_{N_h * N_n} = \Box_{N_n} \circ \Box_{N_h}.
\]
Then, by Proposition~3.20 of \cite{CG} and since $h(d(a)) = n(h(a))$ for all $a \in A$, we obtain:
\begin{align*}
\Box_{N_n}(\Box_{N_h}(\beta_\mathbf{A}(a)))
&= \Box_{N_n}(\beta_\mathbf{B}(h(a))) \\
&= \beta_\mathbf{B}(n(h(a))) \\
&= \beta_\mathbf{B}(h(d(a))) \\
&= \Box_{N_h}(\beta_\mathbf{A}(d(a))) \\
&= \Box_{N_h}(\Box_{N_d}(\beta_\mathbf{A}(a))).
\end{align*}
\end{proof}
 \begin{proposition}
     Let $\langle X,\mathcal{K}_X, T_X\rangle$, $\langle Y,\mathcal{K}_Y, T_Y\rangle$ be two MoS-spaces and let $M\subseteq X\times Y$ be a compatible meet-relation. Then, for all $V \in S(Y)$,
     \begin{center}
         $\Box_M(\Box_{T_Y}(V))=\Box_{T_X}(\Box_M(V))$.
     \end{center}
 \end{proposition}
 \begin{proof}
     Immediate from Lemma 3.21 and Proposition 3.22 of \cite{CG}.
 \end{proof}

\begin{proposition}\label{x(s(X)) es RS espacio}
    Let $\langle X,\mathcal{K},T\rangle$ a MoS-space, then $\langle\mathcal{X}(\mathbf{S}(X)),\mathcal{K}_{\mathbf{S}(X)}, N_{\Box_T}\rangle$ is a MoS-space.
\end{proposition}
\begin{proof}
    The proof follows directly from Remark \ref{definicion D_m} and Proposition \ref{S(X) es modal}.
\end{proof}

We are now in a position to relate modal semilattices and MoS-spaces categorically. Concretely, let $\langle X, \mathcal{K}, T \rangle$ be an MoS-space. By Proposition~\ref{S(X) es modal}, the structure $\langle \mathbf{S}(X), \Box_T \rangle$ is a modal semilattice. Moreover, if $M \subseteq X_1 \times X_2$ is a compatible meet-relation, then $\Box_M \colon \mathbf{S}(X_2) \to \mathbf{S}(X_1)$ is a modal homomorphism. By Theorem~3.24 of \cite{CG}, $\Box_{\sqsupseteq_X} = \mathrm{id}_{\mathbf{S}(X)}$ for every $S$-space $X$. Hence, by Proposition~3.16 of \cite{CG}, the assignments
\[
\begin{array}{ccl}
   \langle X, \mathcal{K}, T \rangle  & \mapsto & \langle \mathbf{S}(X), \Box_T \rangle \\ 
   \\
    M \subseteq X_1 \times X_2 & \mapsto &  \Box_M \colon S(X_2) \to S(X_1)
\end{array}
\]

define a functor $\mathbf{F} \colon \mathsf{MoSsp} \to \mathsf{MoSL}^{op}$.

Conversely, by Proposition~\ref{R_h preserva adj derecha} and Remark~\ref{definicion D_m}, the assignments
\[
\begin{array}{ccl}
\langle\mathbf{A},m\rangle & \mapsto & \langle \mathcal{X}(\mathbf{A}), \mathcal{K}_{\mathbf{A}}, N_m \rangle \\ \\
h \colon \mathbf{A} \to \mathbf{B} & \mapsto & N_h \subseteq \mathcal{X}(\mathbf{B}) \times \mathcal{X}(\mathbf{A})
\end{array}
\]
define a functor $\mathbf{G} \colon \mathsf{MoSL}^{op} \to \mathsf{MoSsp}$.
Finally, from Lemma 9 in \cite{P W}, the map $H_{X} : X \to \mathcal{X}(\mathbf{S}(X))$ defined as in \eqref{eq: H_X} is an isomorphism of $mS$-spaces. Therefore, we obtain:

\begin{proposition}\label{iso rs espacios}
Let $\langle X, \mathcal{K}, T \rangle$ be an MoS-space. Define the relation $I_X \subseteq X \times \mathcal{X}(\mathbf{S}(X))$ by
\[
(x,H_X(y)) \in I_X
\quad \Longleftrightarrow \quad
H_X(x)\subseteq H_X(y).
\]
Then $I_X$ is an isomorphism between the MoS-spaces $\langle X, \mathcal{K}, T \rangle$ and $\langle \mathcal{X}(\mathbf{S}(X)), \mathcal{K}_{\mathbf{S}(X)}, N_{\Box_T}\rangle$.    
\end{proposition}
\begin{proof}
Observe that, by Proposition~3.23 of \cite{CG}, $I_X$ is a meet-relation. Moreover, by Proposition~\ref{x(s(X)) es RS espacio}, the structure $\langle \mathcal{X}(\mathbf{S}(X)), \mathcal{K}_{S(X)}, N_{\Box_T} \rangle$ is an MoS-space, and by Proposition~3.24 of \cite{CG} we have
\[
N_{\Box_T} * I_X = I_X * T.
\]

Thus, $I_X$ is a compatible meet-relation. Let $I_X^{-1} \subseteq \mathcal{X}(\mathbf{S}(X)) \times X$ be given by
\[
(H_X(y),x) \in I_X^{-1} \iff H_X(y) \subseteq H_X(x).
\]
Again by Proposition~3.23 of \cite{CG}, it holds that
\[
I_X^{-1} * I_X = \sqsupseteq_X
\quad \text{and} \quad
I_X * I_X^{-1} = \sqsupseteq_{\mathcal{X}(\mathbf{S}(X))}.
\]

Now we show that $I_X^{-1}$ is a compatible meet-relation. Indeed,  using the previous identities and the associativity of $*$, we obtain
\begin{align*}
I_X^{-1} * N_{\Box_T}
&= I_X^{-1} * (N_{\Box_T} * \sqsupseteq_{\mathcal{X}(\mathbf{S}(X))}) \\
&= I_X^{-1} * (N_{\Box_T} * I_X * I_X^{-1}) \\
&= (I_X^{-1} * I_X) * T * I_X^{-1} \\
&= T * I_X^{-1}.
\end{align*}
Hence, the MoS-spaces $\langle X, \mathcal{K}, T \rangle$ and $\langle \mathcal{X}(\mathbf{S}(X)), \mathcal{K}_{\mathbf{S}(X)}, N_{\Box_T}\rangle$ are isomorphic, as desired.
\end{proof}
\begin{proposition}\label{iso modal semir}
    Let $\langle A,d\rangle$ be a modal semilattice. Then, $\beta \colon A \rightarrow S(\mathcal{X}(\mathbf{A}))$ is an isomorphism of modal semilattices. 
\end{proposition}
\begin{proof}
First, note that by Proposition 3.9 of \cite{CG}, $\beta$  is an isomorphism of semilattices and by Proposition 3.20 of \cite{CG}, it holds that $\beta(d(a))=\Box_{N_d}(\beta(a))$, for all $a \in A$. Therefore, by Proposition \ref{S(X(A)) es modal}, $\beta$ is an isomorphism in $\mathsf{MoSL}$. 
\end{proof}

Propositions \ref{iso rs espacios} and \ref{iso modal semir} show that the functors $\mathbf{F}$ and $\mathbf{G}$ are mutually inverse up to natural isomorphism. Therefore, we obtain the following duality theorem.

\begin{theorem}\label{dualidad RS espaces}
The categories $\mathsf{MoSL}$ and $\mathsf{MoSsp}$ are dually equivalent.
\end{theorem}

\subsection{Relational Duality for SLata}\label{Relational Duality for SLata}
Observe that every SLata $\langle \mathbf{A}, i, d \rangle$ has an associated modal reduct, namely the modal semilattice $\langle \mathbf{A}, d \rangle$. Hence, the duality developed in the previous subsection provides a natural relational semantics for the right adjoint component of an SLata. Our aim is to extend this duality to the full SLata structure. Thus, the key problem is to determine which additional relational conditions encode the left adjoint operator.
\medskip

In order to characterize homomorphism between modal reducts of SLatas that preserve the left adjoint, we first recall a well-known criterion for the existence of left adjoints in order-theoretic terms.

\begin{lemma}\label{existencia izquierda}
Let $P$ and $Q$ be posets and let $d \colon P \to Q$ be an order-preserving map. The following conditions are equivalent:
\begin{itemize}
    \item[(i)] There exists a map $i \colon Q \to P$ such that $i \dashv d$;
    \item[(ii)] For every $q \in Q$, there exists $p \in P$ such that $d^{-1}[\uparrow q]=\uparrow p$.
\end{itemize}
Moreover, for every $q\in Q$, $d^{-1}[\uparrow q]=\uparrow i(q)$.
\end{lemma}

\begin{proof}
The proof is analogous to that of Lemma~7.3.2 in \cite{DP2002}.
\end{proof}

In light of Lemma~\ref{existencia izquierda} one may think that preservation of left adjoints can be characterized through inverse images of principal upsets. The next result shows that, for homomorphisms between modal reducts of $\mathsf{SLatas}$, this condition is in fact equivalent to preserving the left adjoint operation.
    
\begin{lemma}\label{homo slata}
    Let $\langle \mathbf{A}, i_{\mathbf{A}}, d_{\mathbf{A}}\rangle$ and  $\langle \mathbf{B}, i_{\mathbf{B}}, d_{\mathbf{B}}\rangle$ be two SLatas. Let $f:\mathbf{A}\to \mathbf{B}$ be a semilattice homomorphism that preserves right adjoints. Then, for every $a\in A$, the following are equivalent: 
    \begin{enumerate}
        \item $f(i_\mathbf{A}(a))=i_{\mathbf{B}}(f(a)),$
        \item $d_\mathbf{B}^{-1}[\uparrow f(a)]\subseteq\uparrow f(i_\mathbf{A}(a))$
    \end{enumerate}
\end{lemma}
\begin{proof}
Assume first that
$f(i_\mathbf{A}(a))=i_\mathbf{B}(f(a))$
for every $a\in A$. Let
$b\in d_\mathbf{B}^{-1}[\uparrow f(a)]$.
Then
$f(a)\leq d_\mathbf{B}(b)$,
and therefore
$$
f(i_\mathbf{A}(a))
=
i_\mathbf{B}(f(a))
\leq b.
$$
Hence,
$b\in \uparrow f(i_\mathbf{A}(a))$,
which proves (2). Conversely, assume that
$d_\mathbf{B}^{-1}[\uparrow f(a)]
\subseteq
\uparrow f(i_\mathbf{A}(a))$.
By Lemma~\ref{existencia izquierda},
$$
d_\mathbf{B}^{-1}[\uparrow f(a)]
=
\uparrow i_\mathbf{B}(f(a)).
$$
Thus,$\uparrow i_\mathbf{B}(f(a))
\subseteq
\uparrow f(i_\mathbf{A}(a))$,
and consequently
$$
f(i_\mathbf{A}(a))
\leq
i_\mathbf{B}(f(a)).
$$

For the converse inequality, observe that since
$i_\mathbf{A}(a)\leq i_\mathbf{A}(a)$,
we have $a\leq d_\mathbf{A}(i_\mathbf{A}(a))$. Applying $f$ and using that $f$ preserves right adjoints, we obtain
$f(a)\leq d_\mathbf{B}(f(i_\mathbf{A}(a)))$.
Hence,
$$
f(i_\mathbf{A}(a))
\in
d_\mathbf{B}^{-1}[\uparrow f(a)].
$$
By Lemma~\ref{existencia izquierda}, $d_\mathbf{B}^{-1}[\uparrow f(a)]=\uparrow i_\mathbf{B}(f(a)),$
so $i_\mathbf{B}(f(a))\leq f(i_\mathbf{A}(a)).$

Therefore,
$$
f(i_\mathbf{A}(a))
=
i_\mathbf{B}(f(a)).
$$
\end{proof}

The previous results suggest that the existence and preservation of left adjoints should admit a relational reformulation. Recall that the duality developed in the previous subsection applies to the modal reduct of every SLata, namely the semilattice equipped with its right adjoint operator. Therefore, in order to extend this duality to the full SLata structure, we must identify the relational conditions that characterize the existence of a left adjoint for the operator induced by a relation.

By Lemma~\ref{existencia izquierda}, the existence of a left adjoint for an order-preserving map is determined by the inverse images of principal upsets. Our goal is to translate this condition into the setting of $S$-spaces.

Accordingly, if $\langle X,\mathcal{K}\rangle$ is an $S$-space and $U\in S(X)$, we define
\begin{equation}\label{conjunto D_U}
    D_U=\{Y\in \mathcal{C}_{\mathcal{K}}(X)\colon Y\subseteq U\}.
\end{equation}

The sets $D_U$ will allow us to express relationally the principal upsets associated with inverse images under modal operators.

At the level of morphisms, this raises the following question. Suppose that
\[
f:\langle \mathbf{A},d_\mathbf{A}\rangle\to \langle \mathbf{B},d_\mathbf{B}\rangle
\]
is a homomorphism between the modal reducts of two SLatas. Under which conditions does $f$ also preserve the corresponding left adjoint operators?

In general, this does not occur, as the following example shows.
   
\begin{example} 
Let $X=\{0,1\}$ and consider its power set $\mathcal{P}(X)$. Let $r \colon X \to X$ be the constant map given by $r(0)=r(1)=0$. Set $f=r$. Then $\langle \mathbf{A}, f_{\ast}, f^{\ast} \rangle$ is an SLata, where $\mathbf{A}=\langle \mathcal{P}(X), \cap, X \rangle$, $f_{\ast}(Z)=f[Z]$, and $f^{\ast}(Z)=f^{-1}[Z]$ for all $Z \subseteq X$. Observe that $h = r^{-1} \colon \mathcal{P}(X) \to \mathcal{P}(X)$ is a semilattice endomorphism of $\mathbf{A}$ that preserves $f^{\ast}$. However, it does not preserve $f_{\ast}$. Indeed, taking $Z=\{1\}$, we have
\[
h(f_{\ast}(Z)) = X \quad \text{and} \quad f_{\ast}(h(Z)) = \emptyset.
\]
\end{example}

Thus, preserving the right adjoint does not suffice to preserve the left adjoint, so relations that
only encode the right adjoint cannot represent morphisms of SLatas. Motivated by this and by
Lemma \ref{existencia izquierda}, we now introduce the class of relations for which the inverse images
under modal operators are again represented by subbasic closed sets, which in turn guarantees the
existence of a compatible left adjoint.

\begin{definition}\label{A relacion}
    Let $\langle X, \mathcal{K}\rangle$ be an S-space. We say that a relation $T \subseteq X \times X$ is an adjoint-relation (A-relation, for short) if it satisfies the following conditions:
    \begin{enumerate}
        \item $T$ is a meet-relation.
        \item For all $U \in S(X)$, $\bigcap F_{U}^{T} \in S(X)$, where 
        \[
            F_U^{T} = \{ V \in S(X) \colon U \subseteq \Box_{T}(V) \}.
        \]
    \end{enumerate}
\end{definition}

Observe that from \eqref{conjunto D_U}, condition (2) of Definition \ref{A relacion} can be restated as 
\[\bigcap \{V \in S(X) \colon U\in D_{\Box_T(V)}\} \in S(X)\]
for all $U \in S(X)$.

\begin{definition}
A structure $\langle X, \mathcal{K}, T \rangle$ is called a RelS-space if $\langle X, \mathcal{K} \rangle$ is an S-space and $T \subseteq X \times X$ is an $A$-relation.
\end{definition}

Now, we show that A-relations are precisely the meet-relations for which 
the operator $\Box_T$ admits a left adjoint on $\mathbf{S}(X)$.

\begin{proposition}\label{def box T^{*}}
    If $\langle X, \mathcal{K}\rangle$ is a S-space and $T\subseteq X \times X$ is an A-relation, then there exists an operator $\Box_{T}^{*}$ such that $\Box_{T}^{*} \dashv \Box_{T}$. Hence, if $\langle X, \mathcal{K}, T \rangle$ is a RelS-space, then $\langle \mathbf{S}(X), \Box_{T}^{*},\Box_{T}\rangle$ is an SLata.
\end{proposition}

\begin{proof}
Since $T$ is an $A$-relation, by Definition~\ref{A relacion} (1) and Theorem \ref{theo: dualidad semireticulos}, we have that $\langle \mathbf{S}(X),\Box_{T}\rangle$ is a modal semilattice. Moreover, for all $U\in S(X)$, $\bigcap F_U^{T} \in S(X)$. Observe that this is equivalent to saying that $\Box_T^{-1}[\uparrow U]=\uparrow \bigcap F_U^{T}$. Then, by Lemma~\ref{existencia izquierda}, it follows that the operator $\Box_{T}^{\ast}\colon S(X)\to S(X)$ defined by $\Box_{T}^{\ast}(U)=\bigcap F_U^{T}$ is left adjoint to $\Box_T$. 
\end{proof}

\begin{remark}\label{interseccion}
Let $\langle X, \mathcal{K}, T \rangle$ be a RelS-space and let $V \in S(X)$. By Proposition~\ref{def box T^{*}}, we have
\[
\Box_T^{*}(V)=\bigcap F_V^T.
\]
Hence, the family of closed sets determined by $\Box_T^{*}(V)$ can be described as the intersection of the families associated with the elements of $F_V^T$, namely,
\[
D_{\Box_{T}^{*}(V)}
=
\bigcap_{W \in F^{T}_V} D_W.
\]

Indeed, for every $Y\in \mathcal{C}_{\mathcal{K}}(X)$,
\[
Y \in D_{\Box_{T}^{*}(V)}
\iff
Y \subseteq \bigcap F_V^T
\iff
(\forall W\in F_V^T)\; Y\subseteq W
\iff
Y\in \bigcap_{W\in F_V^T} D_W.
\]
\end{remark}

Observe that what Remark \ref{interseccion} reveals, is that preservation of left adjoints admits a purely relational formulation. We may therefore strengthen the notion of morphism between MoS-spaces so as to capture, at the dual level, homomorphisms of SLatas. The following definition introduces the appropriate notion of morphism between RelS-spaces.

\begin{definition}\label{relacion apropiada}
    Let $\langle X, \mathcal{K}_X, T_X \rangle$ and $\langle Y, \mathcal{K}_Y, T_Y \rangle$ be RelS-spaces. We say that a relation  $M\subseteq X\times Y$ is adjoint-preserving if the following conditions hold: 
    \begin{enumerate}
        \item $M$ is a meet-relation.
        \item $M * T_X \;=\; T_Y * M $,
        \item For every $x\in X$, $U\in S(X)$ and $V\in S(Y)$, the following condition holds: 
        \[(M(x) \in D_V \Rightarrow T_X(x)\in D_U)\Rightarrow ( M(x) \in \bigcap_{W \in F^{T_Y}_V}D_W \Rightarrow \uparrow x \in D_U).\]
    \end{enumerate}
\end{definition}

Conditions (1) and (2) ensure that $M$ behaves as a morphism for the modal structure already studied in the previous subsection, namely, that the associated map
\[
\Box_M:S(Y)\to S(X)
\]
preserves the operators $\Box_{T_Y}$ and $\Box_{T_X}$.
Condition (3) is the additional requirement ensuring preservation of the left adjoint operators. The next lemma shows that these conditions exactly characterize homomorphisms of SLatas.

\begin{lemma}\label{funtor morfirmo M}
Let $\langle X, \mathcal{K}_X, T_X \rangle$ and $\langle Y, \mathcal{K}_Y, T_Y \rangle$ be RelS-spaces. Then,  $M\subseteq X\times Y$ is adjoint-preserving if and only if $\Box_{M}:S(Y)\to S(X)$ is a SLata homomorphism between the SLatas $\langle \mathbf{S}(Y), \Box_{T_{Y}}^{\ast}, \Box_{T_{Y}}\rangle$ and $\langle \mathbf{S}(X), \Box_{T_{X}}^{\ast}, \Box_{T_{X}}\rangle$.
\end{lemma}

\begin{proof}

 Using Lemma~\ref{homo slata}, both directions reduce to relating condition (3) of Definition~\ref{relacion apropiada} with inclusion \eqref{ecua0} through the sets $D_U$.
 
Let us assume that $M \subseteq X \times Y$ is an adjoint-preserving relation between the RelS-spaces $\langle X, \mathcal{K}_X, T_X \rangle$ and $\langle Y, \mathcal{K}_Y, T_Y \rangle$. First, note that by Theorem \ref{dualidad RS espaces}, $\Box_M$ is a modal homomorphism between $\langle \mathbf{S}(Y), \Box_{T_Y}\rangle$ and $\langle \mathbf{S}(X), \Box_{T_X}\rangle$. Therefore,
\[
\Box_M(\Box_{T_Y}(V)) = \Box_{T_X}(\Box_M(V)).
\]
Now, in order to prove that $\Box_{M}$ preserves left adjoints, we will use Lemma \ref{homo slata}. Thus, it suffices to prove that
\begin{equation}\label{ecua0}
\Box_{T_X}^{-1}[\uparrow \Box_M(V)] \subseteq \uparrow \Box_M(\Box_{T_Y}^{*}(V)),
\end{equation}
holds for all $V\in S(Y)$. To do so, let $U \in S(X)$ be such that $U \in \Box_{T_X}^{-1}(\uparrow \Box_M(V))$.
Then we get $\Box_M(V) \subseteq \Box_{T_X}(U)$. So, by \eqref{conjunto D_U}, the latter is equivalent to saying that if $M(x) \in D_V$ then $T_X(x) \in D_U$. Since $M$ is adjoint-preserving, by Definition \ref{relacion apropiada}(3) together with \eqref{conjunto D_U}, we obtain that the implication 
\begin{equation}\label{ecua1}
M(x) \in \bigcap_{W \in F^{T_Y}_V} D_W \Rightarrow \uparrow x \in D_U.    
\end{equation}
holds.

It only remains to prove that $U \in \uparrow \Box_M(\Box_{T_Y}^{*}(V))$. To do so, let $ x \in \Box_M(\Box_{T_Y}^{*}(V))$. Then $M(x) \subseteq \Box_{T_Y}^{*}(V)$. By Remark \ref{interseccion} and \eqref{ecua1}, we obtain $\uparrow x \in D_U$, hence $x \in U$. Therefore,
\[
\Box_M(\Box_{T_Y}^{*}(V)) \subseteq U,
\]
so we conclude that \eqref{ecua0} holds. Hence, by Lemma \ref{homo slata}, $\Box_M$ preserves left adjoints, as desired.
\medskip

Conversely, let us assume that $\Box_M : S(Y) \to S(X)$ is an SLata homomorphism between the SLatas $
\langle S(Y), \Box_{T_Y}^{*}, \Box_{T_Y} \rangle$
and $\langle S(X), \Box_{T_X}^{*}, \Box_{T_X} \rangle$. By Theorem \ref{dualidad RS espaces}, Definition \ref{relacion apropiada}(1) and (2) hold. It only remains to show that condition (3) of the aforementioned Definition follows.

Let $x \in X$, $U \in S(X)$, and $V \in S(Y)$ be such that
\[
M(x) \in D_V \Rightarrow T_X(x) \in D_U.
\]
Then, it follows that $\Box_M(V) \subseteq \Box_{T_X}(U)$. Now suppose that $ M(x) \in \bigcap_{W \in F^{T_Y}_V} D_W$.
By Remark \ref{interseccion}, we have $M(x) \in D_{\Box_{T_Y}^{*}(V)}$, so
$M(x) \subseteq \Box_{T_Y}^{*}(V)$ and therefore $x \in \Box_M(\Box_{T_Y}^{*}(V))$. Since $\Box_M$ is an SLata homomorphism, we get $x \in \Box_{T_X}^{*}(\Box_M(V))$. But, from the assumption, $\Box_M(V) \subseteq \Box_{T_X}(U)$, so it follows that $x \in \Box_{T_X}^{*}(\Box_{T_X}(U))$. Finally, since $\langle S(X), \Box_{T_X}^{*}, \Box_{T_X} \rangle$ is a SLata, we conclude that $x \in U$, hence $\uparrow x \in D_U$. This concludes the proof.
\end{proof}

\begin{proposition}\label{composition adjoint-preserving}
    Let $\langle X_j, \mathcal{K}_j, T_j\rangle$, for $j = 1, 2, 3$, be RelS-spaces, and let $M_1 \subseteq X_1 \times X_2$ and $M_2 \subseteq X_2 \times X_3$ be adjoint-preserving relations. Then $M_2 * M_1$ is also an adjoint-preserving relation. Moreover, for every RelS-space $\langle X, \mathcal{K}, T \rangle$, the dual of the specialization order $\sqsupseteq_X \subseteq X \times X$ is an adjoint-preserving relation.
\end{proposition}

\begin{proof}

By Lemma~\ref{homo slata} it suffices to show that $\Box_{M_2 * M_1}=\Box_{M_1}\circ\Box_{M_2}$ intertwines the left adjoints, which follows by applying the adjoint-preserving property first of $M_2$ and then of $M_1$.

Observe that, by Lemma \ref{homo slata}, we need to prove that $\Box_{M_2 * M_1}$ is a SLata homomorphism between $\langle S(X_3), \Box_{T_3}^{*},\Box_{T_3}\rangle$ and $\langle S(X_1), \Box_{T_1}^{*},\Box_{T_1}\rangle$. Accordingly, we will show that 
\[
\Box_{T_1}^{-1}[\uparrow \Box_{M_2 * M_1}(Z)]\subseteq \uparrow (\Box_{M_2 * M_1})(\Box_{T_3}^{*}(Z)),
\] 
for all $Z \in S(X_3)$. First, we recall that $\Box_{M_2 * M_1}=\Box_{M_1}\circ\Box_{M_2}$. Now note that, since $M_1$ and $M_2$ are adjoint-preserving, we have that $\Box_{M_2}(\Box_{T_3}^{*}(Z))=\Box_{T_2}^{*}(\Box_{M_2}(Z))$ and $\Box_{M_1}(\Box_{T_2}^{*}(V))=\Box_{T_1}^{*}(\Box_{M_1}(V))$, for all $Z\in S(X_3)$ and $V\in S(X_2)$. Now, consider $U \in S(X_1)$ and $Z \in S(X_3)$ such that $\Box_{T_1}(U) \in \uparrow \Box_{M_2 * M_1}(Z)$. Then, we have that $\Box_{M_1}(\Box_{M_2}(Z))\subseteq \Box_{T_1}(U)$. Then, since $\Box_{T_1}^{*}\dashv \Box_{T_1}$, we can write 
\[
\Box_{T_1}^{*}(\Box_{M_1}(\Box_{M_2}(Z)))=\Box_{T_1}^{*}(\Box_{M_2 * M_1}(Z)) \subseteq U.
\]
Hence, $U \in \uparrow (\Box_{M_1}\circ \Box_{M_2})(\Box_{T_3}^{*}(Z))$. The moreover part, follows from Lemma \ref{funtor morfirmo M} and Theorem \ref{dualidad RS espaces}.
\end{proof}

We are ready to define a new category, denoted by $\mathsf{RelSp}$. Its morphisms are the so-called adjoint-preserving relations, the identity morphisms are given by the dual of the specialization order, and composition is defined as in \eqref{Definicion composition meet-relations}.
\medskip

Next, we turn to the object part of the algebraic side of the duality. 
The following lemma shows that every SLata gives rise to a RelS-space.

\begin{lemma}\label{X(A) es rel s space}
    If $\langle \mathbf{A}, i, d \rangle$ is an SLata, then $\langle \mathcal{X}(\mathbf{A}), \mathcal{K}_\mathbf{A}, N_d \rangle$ is a RelS-space.
\end{lemma}
\begin{proof}
From Theorem~\ref{theo: dualidad semireticulos}, $\langle \mathcal{X}(\mathbf{A}), 
\mathcal{K}_\mathbf{A}\rangle$ is an S-space. We show that $N_d$ is an A-relation. 
By Remark~\ref{definicion D_m}, $N_d$ is a meet-relation, so 
Definition~\ref{A relacion}(1) holds. For Definition~\ref{A relacion}(2), 
let $a\in A$. We show that $F_{\beta(a)}^{N_d}=\uparrow \beta(i(a))$. 
Indeed, for every $b\in A$:
\begin{displaymath}
    \begin{array}{rcll}
         \beta(b)\in F_{\beta(a)}^{N_d} & \Leftrightarrow & \beta(a)\subseteq \Box_{N_d}(\beta(b)) & \\
         & \Leftrightarrow & \beta(a)\subseteq \beta(d(b)), & \text{by Proposition~3.20 of \cite{CG}}\\
         & \Leftrightarrow & a\leq d(b), &  \\
         & \Leftrightarrow & i(a)\leq b, & \text{since } i\dashv d, \\
         & \Leftrightarrow & \beta(i(a))\subseteq \beta(b), & \text{by Proposition~3.9 of \cite{CG}.}
    \end{array}
\end{displaymath}
Therefore, $\bigcap F_{\beta(a)}^{N_d}=\beta(i(a))\in S(\mathcal{X}(\mathbf{A}))$, 
so Definition~\ref{A relacion}(2) holds.
\end{proof}

The previous lemma guarantees that $\Box_{N_d}^*$ is well defined on 
$\langle \mathcal{X}(\mathbf{A}), \mathcal{K}_\mathbf{A}, N_d\rangle$. 
The following result shows that it precisely captures the left adjoint 
operation of the SLata.

\begin{lemma}\label{beta(i(a))}
   
Let $\langle \mathbf{A},i,d\rangle$ be an SLata. Then, for every $a \in A$, the following holds:
\[
\beta(i(a))=\Box_{N_d}^{\ast}(\beta(a)).
\]
\end{lemma}
\begin{proof}
    This is straightforward from Lemma \ref{X(A) es rel s space} and Proposition \ref{def box T^{*}}. 
\end{proof}

We now consider the morphism part of the construction. The next result shows that homomorphisms of SLatas correspond, at the dual level, to adjoint-preserving relations between the associated RelS-spaces.

\begin{lemma}\label{R_h apropiada}
Let $\langle \mathbf{A}, i_{\mathbf{A}}, d_{\mathbf{A}} \rangle$ and 
$\langle \mathbf{B}, i_{\mathbf{B}}, d_{\mathbf{B}} \rangle$ be two SLatas. 
If $h:\mathbf{A}\to \mathbf{B}$ is an SLata homomorphism, then 
$N_h \subseteq \mathcal{X}(\mathbf{B}) \times \mathcal{X}(\mathbf{A})$ 
is an adjoint-preserving relation between the RelS-spaces 
$\langle \mathcal{X}(\mathbf{B}), \mathcal{K}_{\mathbf{B}}, N_{d_{\mathbf{B}}} \rangle$ 
and 
$\langle \mathcal{X}(\mathbf{A}), \mathcal{K}_{\mathbf{A}}, N_{d_{\mathbf{A}}} \rangle$.
\end{lemma}

\begin{proof}
    Let $\langle \mathbf{A}, i_{\mathbf{A}}, d_\mathbf{A}\rangle$ and $\langle \mathbf{B}, i_{\mathbf{B}}, d_\mathbf{B}\rangle$ be two SLatas, and consider $h:\mathbf{A}\to \mathbf{B}$ a SLata homomorphism. We want to prove that $N_h$ is adjoint-preserving, so, according to Lemma \ref{funtor morfirmo M}, we only need to prove that $\Box_{N_h}$ is a homomorphism of SLatas. In order to prove our claim, by Lemma \ref{homo slata} it suffices to show that for all $a\in A$,
\[
\Box_{N_{d_B}}^{-1}[\uparrow\Box_{N_h}(\beta_A(a))]\subseteq\uparrow\Box_{N_h}(\Box_{N_{d_A}}^{*}(\beta_A(a)))
\]
holds. Let $a \in A,b \in B$, and assume that $\beta_B(b) \in \Box_{N_{d_B}}^{-1}[\uparrow\Box_{N_h}(\beta_A(a))]$. Then, $\Box_{N_h}(\beta_A(a)) \subseteq \Box_{N_{d_B}}(\beta_B(b))$. Since $\Box_{N_h}(\beta_A(a))=\beta_B(h(a))$ and $\Box_{N_{d_B}}^{*}\dashv \Box_{N_{d_B}}$, the latter, together with Lemma \ref{beta(i(a))}, yields $\Box_{N_{d_B}}^{*}(\beta_B(h(a)))\subseteq \beta_B(b)$. Thus, from the assumption on $h$, we obtain $\Box_{N_h}(\Box_{N_{d_A}}^{*}(\beta_A(a))) \subseteq \beta_B(b)$, which shows that $\beta_B(b) \in \uparrow\Box_{N_h}(\Box_{N_{d_A}}^{*}(\beta_A(a)))$. So, by Lemmas \ref{homo slata} and \ref{dualidad RS espaces}, we have that $\Box_{N_h}$ is a SLata homomorphism. Then, by Lemma \ref{funtor morfirmo M} we can conclude that $N_h$ is adjoint-preserving.
\end{proof}

We stress that lemmas \ref{X(A) es rel s space} and \ref{R_h apropiada} provide the object and morphism components of the relational representation of SLatas. We now show that this construction faithfully recovers the original algebraic structure.

\begin{proposition}\label{iso S-Lata}
      Let $\langle\mathbf{A},i,d\rangle$ be an SLata. Then, $\langle \mathbf{S}(\mathcal{X}(\mathbf{A})), \Box_{N_d}^{*},\Box_{N_d}\rangle$ is an SLata and the map $\beta$ is an isomorphism of SLatas.
 \end{proposition}
     \begin{proof}
         First, note that by Lemma \ref{X(A) es rel s space} and Proposition \ref{def box T^{*}} we know that $\langle \mathbf{S}(\mathcal{X}(\mathbf{A})), \Box_{N_d}^{*},\Box_{N_d}\rangle$ is an SLata. Moreover, from Proposition \ref{iso modal semir} and Lemma \ref{beta(i(a))} it holds that $\beta(d(a))=\Box_{N_d}(\beta(a))$ and $\beta(i(a))=\Box_{N_d}^{*}(\beta(a))$. 
     \end{proof}

We now establish the converse reconstruction at the relational level. More precisely, we show that the canonical space associated with the SLata of subbasic closed sets of a RelS-space again carries a natural RelS-space structure.

\begin{lemma}\label{x(s(X)) es rel space}
    Let $\langle X, \mathcal{K}, T \rangle$ be a RelS-space. Then, $\langle \mathcal{X}(\mathbf{S}(X)),\mathcal{K}_{\mathbf{S}(X)},N_{\Box_T}\rangle$ is a RelS-space.
\end{lemma}

\begin{proof}
We write $Q = N_{\Box_T}$ for readability. Since $H_X$ is bijective and 
carries subbasic elements of $X$ into subbasic elements of 
$\mathcal{X}(\mathbf{S}(X))$, there is no risk of confusion in assuming 
that every element of $S(\mathcal{X}(\mathbf{S}(X)))$ is of the form 
$H_X[U]$, for some $U\in S(X)$.

We prove that $Q$ is an A-relation. Since $T$ is an A-relation, by 
Definition~\ref{A relacion}(1) it is in particular a meet-relation, so 
$\langle X, \mathcal{K}, T\rangle$ is a MoS-space. By 
Proposition~\ref{x(s(X)) es RS espacio}, 
$\langle \mathcal{X}(\mathbf{S}(X)), \mathcal{K}_{\mathbf{S}(X)}, Q\rangle$ 
is a MoS-space, so $Q$ is a meet-relation and Definition~\ref{A relacion}(1) 
holds.

For Definition~\ref{A relacion}(2), let $U \in S(X)$. Since $T$ is an 
A-relation, $\bigcap F^T_U \in S(X)$. We claim that
\[
\bigcap F^Q_{H_X[U]} = H_X\bigl[\bigcap F^T_U\bigr].
\]
Indeed, by Theorem~\ref{dualidad RS espaces} and the bijectivity of $H_X$, 
we have $F^Q_{H_X[U]} = H_X[F^T_U]$, and thus 
$\bigcap F^Q_{H_X[U]} = \bigcap H_X[F^T_U] = H_X[\bigcap F^T_U] 
\in S(\mathcal{X}(\mathbf{S}(X)))$. 
Hence $Q$ is an A-relation, and $\langle \mathcal{X}(\mathbf{S}(X)), 
\mathcal{K}_{\mathbf{S}(X)}, N_{\Box_T}\rangle$ is a RelS-space.
\end{proof}

The next result plays a fundamental role in the reconstruction process. It shows that, up to isomorphism, the left adjoint operation of an SLata is completely determined by its modal reduct. In particular, any isomorphism between modal reducts automatically preserves the left adjoint structure.

\begin{lemma}\label{iso de modal preserva izquierdo}
Two SLatas $\langle \mathbf{A},i_\mathbf{A},d_\mathbf{A}\rangle$ and $\langle \mathbf{B},i_\mathbf{B},d_\mathbf{B}\rangle$ are isomorphic if and only if their modal reducts $\langle A,d_\mathbf{A}\rangle$ and $\langle B,d_\mathbf{B}\rangle$ are isomorphic.
\end{lemma}
\begin{proof}
On the one hand, let us assume that $f:A\to B$ is an isomorphism between $\langle \mathbf{A},d_\mathbf{A}\rangle$ and $\langle \mathbf{B},d_\mathbf{B}\rangle$. We show that $f^{-1}$ preserves left adjoints. Indeed, let $a\in A$ and $b\in B$. Since $f$ preserves right adjoints, we have:
\[
f^{-1}(i_\mathbf{B}(b))\leq a \;\Leftrightarrow\; i_\mathbf{B}(b)\leq f(a)\;\Leftrightarrow\; b\leq d_\mathbf{B}(f(a))=f(d_\mathbf{A}(a))\;\Leftrightarrow\; f^{-1}(b)\leq d_\mathbf{A}(a).
\]
Therefore, $f^{-1}(i_\mathbf{B}(b)) = i_\mathbf{A}(f^{-1}(b))$. Hence, $f^{-1}$ is an isomorphism in $\mathsf{SLata}$. It is not hard to see that this implies that $f$ also preserves left adjoints. The last implication is straightforward.
\end{proof}

\begin{lemma}\label{iso rel spaces}
The RelS-spaces $\langle X, \mathcal{K}, T \rangle$ and $\langle \mathcal{X}(\mathbf{S}(X)), \mathcal{K}_{\mathbf{S}(X)}, N_{\Box_T} \rangle$ are isomorphic. 
\end{lemma}

\begin{proof}
Consider the relations $I_X\subseteq X \times \mathcal{X}(\mathbf{S}(X))$ and $I_X^{-1}\subseteq \mathcal{X}(\mathbf{S}(X))\times X$, defined as in Proposition \ref{iso rs espacios}. We already know that $I_X^{-1} * I_X = \sqsupseteq_X$ and $I_X * I_X^{-1} = \sqsupseteq_{\mathcal{X}(\mathbf{S}(X))}$. Thus, to prove our claim, it suffices to show that $I_X$ and $I_X^{-1}$ are both adjoint-preserving. To this end, observe that by Proposition \ref{def box T^{*}}, Lemma \ref{x(s(X)) es rel space}, and Theorem \ref{dualidad RS espaces}, the structures $\langle \mathbf{S}(X), \Box_T^{*},\Box_T\rangle$ and $\langle \mathbf{S}(\mathcal{X}(\mathbf{S}(X))), \Box_{N_{\Box_T}}^{*},\Box_{N_{\Box_T}}\rangle$ are SLatas whose modal reducts are isomorphic via $\Box_{I_X}$. Moreover, $\Box_{I_X}^{-1}=\Box_{I_X^{-1}}$. Hence, by Lemma \ref{iso de modal preserva izquierdo}, both $\Box_{I_X}$ and $\Box_{I_X^{-1}}$ are SLata homomorphisms. Therefore, by Lemma \ref{funtor morfirmo M}, both $I_X$ and $I_X^{-1}$ are adjoint-preserving.
\end{proof}

We now use Lemma~\ref{iso rel spaces} to define the functors $\mathbf{M}\colon \mathsf{SLata}^{op}\to\mathsf{RelSp}$ and $\mathbf{R}\colon\mathsf{RelSp}\to\mathsf{SLata}^{op}$ and to show that they are quasi-inverse, in the sense that the composition $\mathbf{R}\circ\mathbf{M}$ is naturally isomorphic to the identity functor of $\mathsf{SLata}^{op}$, and $\mathbf{M}\circ\mathbf{R}$ is naturally isomorphic to the identity functor of $\mathsf{RelSp}$. This leads to the main theorem of this section.

Let $\langle \mathbf{A}_j, i_j, d_j \rangle$, with $j = 1, 2$, be two SLatas, and let $h : \mathbf{A}_1 \to \mathbf{A}_2$ be a SLata homomorphism. From Lemma \ref{R_h apropiada}, we have that $N_h$ is an adjoint-preserving relation. Furthermore, if $h = \mathrm{id}_{\mathbf{A}}$, then $N_h = \sqsupseteq_{\mathcal{X}(\mathbf{A})}$, and from Proposition 3.22 of \cite{CG}, we also have that if $h : A \to B$ and $g : B \to C$ are morphisms of SLatas, then $N_{g h} = N_h \ast N_g$. Based on the latter and lemmas \ref{X(A) es rel s space} and \ref{R_h apropiada}, we obtain that the assignments
\[
\begin{array}{ccl}
   \langle\mathbf{A}, i, d\rangle & \mapsto & \langle \mathcal{X}(\mathbf{A}), \mathcal{K}_\mathbf{A}, N_{d} \rangle, \\
   h: \mathbf{A}\to \mathbf{B} & \mapsto & N_h, 
\end{array}
\]
define a functor \(\mathbf{M} : \mathsf{SLata}^{op} \to \mathsf{RelSp}\).

Conversely, from lemmas~\ref{funtor morfirmo M} and \ref{R_h apropiada} the assignments
\[
\begin{array}{ccl}
    \langle X,\mathcal{K},T\rangle  & \mapsto & \langle \mathbf{S}(X),\Box_{T}^{*},\Box_{T}\rangle, \\
     M\subseteq X_1\times X_2 & \mapsto & \Box_{M}: \mathbf{S}(X_2)\to \mathbf{S}(X_1), 
\end{array}
\]
define a functor $\mathbf{R}:\mathsf{RelSp} \rightarrow \mathsf{SLata}^{op}$.

Moreover, if $M \subseteq X \times Y$ is adjoint-preserving, then it follows from the above that the relation
\[
N_{\Box_M} \subseteq \mathcal{X}(\mathbf{S}(X)) \times \mathcal{X}(\mathbf{S}(Y))
\]
is adjoint-preserving.

Furthermore, since every adjoint-preserving relation between RelS-spaces is a meet-relation, Theorem~\ref{dualidad RS espaces} guarantees that the following diagram commutes for every adjoint-preserving relation $M$ between the RelS-spaces $\langle X,\mathcal{K}_X,T_X\rangle$ and $\langle Y,\mathcal{K}_Y,T_Y\rangle$:
\[
\begin{tikzcd}
\langle X,\mathcal{K}_X,T_X\rangle 
\arrow[r,"I_X"] 
\arrow[d,"M"'] 
& 
\langle \mathcal{X}(\mathbf{S}(X)),\mathcal{K}_{\mathbf{S}(X)},N_{\Box_{T_X}}\rangle 
\arrow[d,"N_{\Box_M}"] \\
\langle Y,\mathcal{K}_{Y},T_Y\rangle 
\arrow[r,"I_Y"'] 
& 
\langle \mathcal{X}(\mathbf{S}(Y)),\mathcal{K}_{\mathbf{S}(Y)},N_{\Box_{T_Y}}\rangle
\end{tikzcd}
\]

Combined with Lemma~\ref{iso rel spaces} and Proposition~\ref{iso S-Lata}, these constructions lead to the main theorem of this section.

\begin{theorem}\label{Relational duality for SLatas}
The categories $\mathsf{SLata}$ and $\mathsf{RelSp}$ are dually equivalent.
\end{theorem}

\vspace{0.5cm}

Theorem~\ref{Relational duality for SLatas} provides a simpler relational representation of SLatas than the multirelational approach of \cite{GPZ}. In the latter, an SLata-space is described by two multirelations \(I,E\subseteq X\times\mathcal P(X)\), which must satisfy (A1)--(A3). In particular, (A1) requires both \(\langle X,K,I\rangle\) and \(\langle X,K,E\rangle\) to be mS-spaces, while (A2)--(A3) impose additional compatibility conditions involving families of subsets of \(X\). In contrast, the present approach represents an SLata using a single binary 
relation $T\subseteq X\times X$ satisfying the conditions of 
Definition~\ref{A relacion}: $T$ is a meet-relation, and 
$\bigcap F_U^T\in\mathbf S(X)$ for every $U\in\mathbf S(X)$. Thus, the 
representation replaces two multirelations by one binary relation and three 
axioms by two conditions on $T$, formulated directly in terms of the S-space 
$\langle X,K\rangle$ and its family $\mathbf S(X)$, without explicitly 
encoding transitions as families of subsets of $X$.

\section{$\mathsf{SLataSp}$ and $\mathsf{RelSp}$ are isomorphic}\label{sec: The isomorphism}

By the results of \cite{GPZ} and Section~\ref{Relational Duality for SLata}, it follows that the categories $\mathsf{SLataSp}$ and $\mathsf{RelSp}$ are equivalent. From a categorical perspective, this already provides a satisfactory correspondence between SLatas and their relational semantics. However, this equivalence is obtained through several intermediate constructions and does not directly identify the underlying spatial and relational structures. As a consequence, the extent to which both notions of space coincide remains partially hidden.
\medskip

The aim of this section is to strengthen this correspondence by showing that $\mathsf{SLata}$ and $\mathsf{RelSp}$ are, in fact, isomorphic categories. This provides a more rigid and canonical identification of the two frameworks, showing that the relational semantics introduced earlier does not merely model SLatas up to equivalence, but reproduces them up to isomorphism of categories. For this purpose, in Subsection \ref{Modal Semilattices and Normal Spaces} we introduce the notion of normal mS-spaces, which serve as the multirelational counterpart of modal semilattices. Then, in Subsection \ref{the isomorphism of categories}, we establish the desired isomorphism of categories.

\subsection{Modal Semilattices and Normal Spaces}\label{Modal Semilattices and Normal Spaces}

In order to refine the correspondence obtained in \cite{GPZ} and Section~\ref{Relational Duality for SLata}, we introduce in this subsection a notion of normality for mS-spaces. This condition isolates the structural content underlying modal semilattices and allows us to recover their associated meet-relational behavior in a direct way. This will serve as the technical bridge needed in the next subsection to establish an isomorphism between $\mathsf{SLataSp}$ and $\mathsf{RelSp}$.

\medskip

In \cite{CM}, a notion of normality plays a central role in the duality theory of distributive semilattices. Roughly speaking, normality allows the relational structure to be recovered from local information in the associated space. This motivates the introduction of an analogous condition in the setting of mS-spaces.

\medskip

We first recall Proposition~\ref{decreciente interseccion de abiertos}, which will be used implicitly in the subsequent discussion in order to properly interpret certain constructions in mS-spaces.

\begin{proposition}\label{decreciente interseccion de abiertos}        
Let $\langle X,\mathcal{K}, R \rangle$ be an $mS$-space. Then, for every $x \in X$, the set 
\[(x]=\{y\in X : y\sqsupseteq x\}\] 
belongs to $\mathcal{Z}(X)$.
\end{proposition}
\begin{proof}
Observe that $(x]=\bigcap \mathcal{L}$, where $\mathcal{L}=\{U\in \mathcal{K}\colon x\in U\}$. The fact that $\mathcal{L}$ is dually directed is straightforward from (S3).
\end{proof}

\begin{definition}\label{espacio normal}
    An mS-space $\langle X, \mathcal{K}, R \rangle$ is called \textit{normal} if:  
\begin{enumerate}
    \item[\normalfont {(N1)}] For any $x \in X$ and for each $Z \in \mathcal{Z}(X)$ such that $Z \in R(x)$, there exists $z \in Z$ such that $(z] \in R(x)$,  
    \item[\normalfont (N2)] $\emptyset \notin R(x)$ for every $x \in X$.
\end{enumerate}

\end{definition}

The notion of normality introduced above is designed to capture precisely the additional structure required to recover modal semilattices within the framework of mS-spaces. The following result shows that this condition is in fact exact: it characterizes modal semilattices at the level of their associated mS-spaces.
\medskip

Let $\mathbf{A}$ be a semilattice. Recall that every modal semilattice $\langle \mathbf{A}, m \rangle$ is, in particular, a monotone semilattice. Consequently, by Theorem \ref{duality Msp mMS}, it induces an $mS$-space $\langle \mathcal{X}(\mathbf{A}), \mathcal{K}_{\mathbf{A}}, R_m \rangle$, where the relation
$R_m \subseteq \mathcal{X}(\mathbf{A}) \times \mathcal{Z}(\mathcal{X}(\mathbf{A}))$
is given by
\begin{equation}\label{def: def R_m}
(P,Z)\in R_m \quad \Longleftrightarrow \quad m^{-1}[P]\cap I_{\mathbf{A}}(Z)=\emptyset,
\end{equation}
where $I_{\mathbf{A}}(Z)= \{a \in A : \beta(a) \cap Z = \emptyset\}$.

\begin{remark}\label{ideal decreciente}
    Let $\langle \mathbf{A}, m \rangle$ a modal semilattice. Then $I_\mathbf{A}((P]) = P^c$ , for $P \in  \mathcal{X}(\mathbf{A})$ in the space $\langle \mathcal{X}(\mathbf{A}), \mathcal{K}_\mathbf{A}, R_m \rangle$.
\end{remark}

We now relate modal semilattices with the introduced notion of normality for $mS$-spaces.

\medskip

The proof of the following result follows the same lines as that of Proposition~70 in \cite{CM}. The main difference is that, in the present setting, saturated subsets are replaced by subbasic saturated subsets.

This adaptation is possible because, for semilattices, the posets
\[
\langle \mathrm{Id}(\mathbf{A}), \subseteq \rangle
\quad\text{and}\quad
\langle \mathcal{Z}(\mathcal{X}(\mathbf{A})), \subseteq \rangle
\]
are dually isomorphic (see \cite{P W}) via the correspondences
\[
Z \longmapsto I_{\mathbf{A}}(Z)
\quad \text{and} \quad
I \longmapsto \alpha(I)=\{P \in \mathcal{X}(\mathbf{A}) : I \cap P = \emptyset\}.
\]

Together with Proposition~\ref{decreciente interseccion de abiertos}, this allows the argument of Proposition~70 in \cite{CM} to be transferred verbatim. We therefore omit the details.

\begin{lemma}\label{dual normal}
Let $\langle \mathbf{A}, m \rangle$ be a monotone semilattice.  
Then, $\langle \mathbf{A}, m \rangle$ is a modal semilattice if and only if the space $\langle \mathcal{X}(\mathbf{A}), \mathcal{K}_\mathbf{A}, R_m \rangle$ is a normal $mS$-space.
\end{lemma}

This lemma is the key bridge of the subsection. It shows that the algebraic notion of a modal semilattice is exactly captured by the spatial condition of normality on the associated $mS$-space.   
\medskip

Now, we proceed to show how we get a normal mS-space from a MoS-space.
\medskip

Let $\langle X, \mathcal{K} \rangle$ be an S-space, and let $T \subseteq X \times X$ be a meet-relation. From Theorem \ref{dualidad RS espaces}, we know that the map $\Box_T: S(X) \rightarrow S(X)$ is a modal operator, so in particular, a monotone operator. Then, by Theorem \ref{duality Msp mMS}, there exists a multirelation
\[
R_{\Box_T}\subseteq \mathcal{X}(\mathbf{S}(X)) \times \mathcal{Z}(\mathcal{X}(\mathbf{S}(X)).
\]
that makes the structure $\langle \mathcal{X}(\mathbf{S}(X)), \mathcal{K}_{\mathbf{S}(X)}, R_{\Box_T}\rangle$ a mS-space. Now, we define a multirelation $R_T \subseteq X \times \mathcal{Z}(X)$ by
\begin{equation}\label{T sombrero}
(x, Z) \in R_T \iff (H_X(x), H_X[Z]) \in R_{\Box_T}.
\end{equation}

The following lemma collects several compatibility properties of the isomorphism $H_X$, which will be repeatedly used in the subsequent constructions. 

\begin{lemma}\label{lem: cosas de H_X}
    Let $\langle X, \mathcal{K}\rangle$ be an S-space and consider the isomorphism $H_X\colon X \rightarrow \mathcal{X}(\mathbf{S}(X))$, defined as in \eqref{eq: H_X}. Then, for all $x\in X$ and $U\in S(X)$, the following hold:
    \begin{enumerate}
        \item $\beta_{\mathbf{S}(X)}(U)=H_X[U]$,
        \item $L_{\beta_{\mathbf{S}(X)}(U)}=H_{X}[L_U]$,
        \item $m_{R_{\Box_T}}(\beta_{\mathbf{S}(X)}(U))=H_X[m_{R_T}(U)]$.
        \item The posets $\langle X,\sqsupseteq\rangle$ and $\langle \mathcal{X}(\mathbf{S}(X)), \subseteq\rangle$ are isomorphic. In particular, $(H_X(x)]=H_X((x])$ for every $x \in X $. 
    \end{enumerate}
\end{lemma}
\begin{proof}
Notice that (1) is immediate from the definitions of $H_X$ and $\beta_{\mathbf{S}(X)}$. Condition (2) follows from (1) and the bijectivity of $H_X$. Condition (3) follows from (2). Finally, (4) follows from the fact that $H_X$ is a homeomorphism and that inclusion in $\mathcal{X}(\mathbf{S}(X))$ coincides with the dual of its specialization order.
\end{proof}

\begin{theorem}\label{theo: MoS-spaces to Normal spaces}
If $\langle X, \mathcal{K}, T \rangle$ is a MoS-space, then, $\langle X, \mathcal{K},  R_T\rangle$ is a normal space. Moreover, \begin{equation}
    (x,Z) \in R_T \iff T(x) \cap Z \neq \emptyset,
\end{equation}
for all $x\in X$ and $Z\in \mathcal{Z}(X)$.
\end{theorem}
\begin{proof}
    Since $\langle \mathcal{X}(\mathbf{S}(X)), \mathcal{K}_{\mathbf{S}(X)}, R_{\Box_T}\rangle$ is a mS-space, by Lemma \ref{lem: cosas de H_X}(2) and (3), it is not hard to see that $\langle X, \mathcal{K}, R_T\rangle$ satisfies conditions (m1) and (m2), and hence it is a mS-space. Now, since $\langle \mathbf{S}(X),\Box_T\rangle$ is a modal semilattice, by Lemma \ref{dual normal} we obtain that $\langle \mathcal{X}(\mathbf{S}(X)), \mathcal{K}_{\mathbf{S}(X)}, R_{\Box_T}\rangle$ is a normal mS-space. Next, we prove that $\langle X, \mathcal{K}, R_T\rangle$ satisfies (N1). To do so, suppose by contradiction that there exist $x\in X$ and $Z\in R_{T}(x)$ such that $(z]\notin R_T(x)$ for every $z\in Z$. Then, without risk of confusion, we may assume that $H_X(x)\in \mathcal{X}(\mathbf{S}(X))$ and $H_X[Z]\in \mathcal{Z}(\mathcal{X}(\mathbf{S}(X)))$ satisfy $H_X((z])\notin R_{\Box_T}(H_X(x))$. By Lemma \ref{lem: cosas de H_X}(4), this yields a contradiction. The proof of (N2) is analogous. Therefore, $\langle X, \mathcal{K}, R_T\rangle$ is a normal space, as claimed.
    \\
    For the moreover part, let $(x,Z)\in R_T$, then by \eqref{T sombrero}, ${\Box_T}^{-1}[(H_X(x)]\subseteq {I_{\mathbf{S}(X)}(H_X[Z])}^{c}$. Notice that from Lemma \ref{lem: cosas de H_X}(2), for all $U\in S(X)$, it is the case that if $T(x)\subseteq U$ then, $U\cap Z\neq \emptyset$. Suppose that $T(x)\cap Z=\emptyset$. Since $Z \in \mathcal{Z}(X)$, there exists a dually directed family $\mathcal{L}\subseteq\mathcal{K}$ such that  $Z=\bigcap\mathcal{L}$. Since $T$ is meet-relation, $T(x)\in \mathcal{C}_{\mathcal{K}}(X)$. So, by Remark \ref{rem: Remark1 CMZ} there
    exists $V\in \mathcal{L}$ such that $T(x)\cap V=\emptyset$. Then, $V^{c}\cap Z\not=\emptyset$. Thus, $V^c\nsubseteq Z^{c}=\bigcup \{V^c\colon V\in \mathcal{L}\}$, which is absurd. So it must be $T(x)\cap Z \neq \emptyset$. Conversely, let $x \in X$ and $Z \in \mathcal{Z}(X)$ be such that $T(x)\cap Z\neq\emptyset$. Let $U \in S(X)$ be such that $U\in {\Box_T}^{-1}[(H_X(x)]$. Then, $T(x)\subseteq U$, and by hypothesis, we obtain $U \cap Z\neq \emptyset$. Therefore, by Lemma \ref{lem: cosas de H_X}(2), it follows that $\beta_{S(X)}(U) \cap H_{X}[Z]\neq \emptyset$. Thus, $U \in {I_{S(X)}(H_X[Z])}^{c}$, and consequently $(H_X(x),H_X[Z]) \in R_{\Box_T}$, which implies $(x,Z) \in R_T$. This concludes the proof.
\end{proof}

We now prove the converse direction, namely that normal mS-spaces give rise to modal semilattices, thus completing the correspondence between both structures. To this end, we need a technical result first.

\begin{lemma}\label{lem: normal to modal}
If $\langle X, \mathcal{K}, R \rangle$ is a normal mS-space, then $\langle \mathbf{S}(X), m_R\rangle$ is a modal semilattice.   
\end{lemma}
\begin{proof}
In order to show that $m_R(X)=X$, observe that $L_X=\mathcal{Z}(X)-\{\emptyset\}$. Thus, by (N2), it follows that $R(x)\subseteq L_X$ for every $x\in X$. On the other hand, since $\langle X, \mathcal{K}, R \rangle$ is, in particular, a monotone space, Theorem \ref{duality Msp mMS} implies that $m_R$ is a monotone operator on $\mathbf{S}(X)$. Hence, for every $U,V\in S(X)$,
\[
m_R(U\cap V)\subseteq m_R(U)\cap m_R(V).
\]
We prove the reverse inclusion. Let $x\in m_R(U)\cap m_R(V)$. Then, for every $Z\in R(x)$, we have $Z\in L_U$ and $Z\in L_V$. By (N1), there exists $z\in Z$ such that $(z]\subseteq L_U$ and $(z]\subseteq L_V$. Hence, there exist $t_U\in U$ and $t_V\in V$ such that $t_U\sqsupseteq z$ and $t_V\sqsupseteq z$. Therefore, $z\in U\cap V$, which implies that $(z]\in L_{U\cap V}$. Moreover, by Proposition \ref{decreciente interseccion de abiertos}, we have $(z]\subseteq Z$. From this it readily follows that $Z\in L_{U\cap V}$. Consequently, $R(x)\subseteq L_{U\cap V}$, and thus
\[
m_R(U)\cap m_R(V)\subseteq m_R(U\cap V).
\]
Therefore, $m_R$ is modal, as desired.
\end{proof}

Let $\langle X, \mathcal{K}, R \rangle$ be a normal mS-space. By Lemma \ref{lem: normal to modal} and Theorem \ref{dualidad RS espaces}, we obtain that $\langle \mathcal{X}(\mathbf{S}(X)), \mathcal{K}_{\mathbf{S}(X)}, N_{m_R}\rangle$ is a MoS-space.  
Now, we define a relation $T_R \subseteq X \times X$ by

\begin{equation}\label{meet asociada a multi}
(x, y) \in T_R \iff (H_X(x),H_X (y)) \in N_{m_R}.
\end{equation}

\begin{lemma}\label{lem: otras cosas de H_X}
Let $\langle X, \mathcal{K}, R \rangle$ be a normal mS-space. Then, for every $U\in S(X)$,
\[\Box_{N_{m_R}}(\beta_{\mathbf{S}(X)}(U))= H_X[\Box_{T_R}(U)].
\]
\end{lemma}
\begin{proof}
Let $H_X(x)\in \Box_{N_{m_R}}(\beta_{\mathbf{S}(X)}(U))$. Then, $N_{m_R}(H_X(x))\subseteq \beta_{\mathbf{S}(X)}(U)$. Thus, for every $H_X(y)\in \mathcal{X}(\mathbf{S}(X))$, if $(H_X(x), H_X(y))\in N_{m_R}$, we have $H_X(y)\in \beta_{\mathbf{S}(X)}(U)$. By the bijectivity of $H_X$, together with Lemma \ref{lem: cosas de H_X}(1) and \eqref{meet asociada a multi}, this is equivalent to saying that $T_R(x)\subseteq U$. Hence, $x\in \Box_{T_R}(U)$. The converse is analogous.
\end{proof}

\begin{theorem}\label{binaria asociada a normal}
If $\langle X, \mathcal{K}, R \rangle$ is a normal mS-space, then  $\langle X, \mathcal{K},T_R\rangle$ is a MoS-space. Furthermore, 
\begin{equation}
(x,y) \in T_R \iff (x,(y]) \in R,
\end{equation}
for all $x,y\in X$.
\end{theorem}
\begin{proof}
We prove that $T_R$ is a meet-relation. To this end, since $\langle \mathcal{X}(\mathbf{S}(X)), \mathcal{K}_{\mathbf{S}(X)}, N_{m_R}\rangle$ is a MoS-space, $N_{m_R}$ is a meet-relation. Thus: (a) for all $U\in S(X)$, there exists $V\in S(X)$ such that $\Box_{N_{m_R}}(\beta_{\mathbf{S}(X)}(U))=\beta_{\mathbf{S}(X)}(V)$; and (b) for every $x\in X$,
\[
N_{m_R}(H_X(x))=\bigcap\{\beta_{\mathbf{S}(X)}(U)\colon H_X(x)\in \Box_{N_{m_R}}(\beta_{\mathbf{S}(X)}(U))\}.
\]

By (a) and Lemma \ref{lem: otras cosas de H_X},
\[
H_X[\Box_{T_R}(U)]
=
\beta_{\mathbf{S}(X)}(V)
\]
for some $V\in S(X)$. Since $\beta_{\mathbf{S}(X)}$ is injective, it follows that
\[
\Box_{T_R}(U)=V\in S(X).
\] To prove the remaining condition, we proceed by contradiction. Assume that there exists
\[
y\in \bigcap\{U\in S(X)\colon x\in \Box_{T_R}(U)\}
\]
such that $y\notin T_R(x)$. Then, by \eqref{meet asociada a multi}, we obtain
\[
H_X(y)\notin N_{m_R}(H_X(x)).
\]
Hence, by (b), there exists $U\in S(X)$ such that
\[
H_X(x)\in \Box_{N_{m_R}}(\beta_{\mathbf{S}(X)}(U))
\]
and
\[
H_X(y)\notin \beta_{\mathbf{S}(X)}(U).
\]
Therefore, by Lemma \ref{lem: otras cosas de H_X} and Lemma \ref{lem: cosas de H_X}(1), we get $x\in \Box_{T_R}(U)$ but $y\notin U,$ a contradiction. This shows that $T_R$ is a meet-relation.
\\
For the last part, let $x,y\in X$ be such that $(x,y)\in T_R$. Then, by Remark \ref{definicion D_m}, we have
\[
m_R^{-1}[H_X(x)]\subseteq H_X(y).
\]

Suppose, for contradiction, that $(y]\notin R(x)$. Since $\langle X, \mathcal{K}, R \rangle$ is a normal mS-space, by (m2) there exists
$U\in S(X)$ such that $x\in m_R(U)$ and $(y]\cap U=\emptyset$. So, in particular, $y\notin U$. Therefore, by the previous inclusion, we get that $m_R(U)\in H_X(x)$, so, by Lemma \ref{lem: cosas de H_X}(1), $y\in U$, a contradiction.

On the other hand, assume that $(x,(y])\in R$ and let $U\in S(X)$ be such that $
U\in m_R^{-1}[H_X(x)]$. Then, $x\in m_R(U)$. Thus by (m2) $(y]\in L_U $. Hence,
$(y]\cap U\neq\emptyset $. Suppose that $y\notin U$. Then $(y]\cap U=\emptyset$, a contradiction. Hence,
$y\in U$, and consequently $
U\in H_X(y)$.

\end{proof}

We are now in a position to close the loop of the constructions developed in this subsection. The next results show that the passage from MoS-spaces to normal $mS$-spaces and back recovers both the binary relation and the induced modal operator. This shows that the two translations are mutually inverse, completing the correspondence between the relational and multirelational structures.

\begin{lemma}\label{T=T_R_T}    
Let $\langle X, \mathcal{K}, T \rangle$ be a MoS-space. Then the following hold:
\begin{enumerate}
    \item $T = T_{R_T}$,
    \item $m_{R_T} = \Box_{T}$.
\end{enumerate}
\end{lemma}
\begin{proof} 
(1) Let $(x,y) \in T$. Since $T$ is a meet-relation, we have that $y \in \bigcap\{U \in S(X)\colon T(x)\subseteq U\}$. To show that $(x,y)\in T_{R_T}$, we will prove that ${\Box_T}^{-1}(H_X(x))\cap {I_{\mathbf{S}(X)}(H_X[(y]])}=\emptyset$. Let $U \in S(X)$ be such that $\Box_T(U) \in H_X(x)$, so we have that $x \in \Box_T(U)$. By definition, we obtain $T(x)\subseteq U$, and by hypothesis we can write $y \in U$. Moreover, since $y \in (y]$, we get $U \cap (y]\neq\emptyset$. Then, by Lemma \ref{lem: cosas de H_X}(1) and (4), we have that $H_X[(y]] \cap \beta_{\mathbf{S}(X)}(U)\neq \emptyset$. Thus, $U \notin {I_{\mathbf{S}(X)}(H_X[(y]])}$. Then, from \eqref{T sombrero} we get $(x,(y]) \in R_T$, so $(x,y) \in T_{R_T}$. 

On the other hand, consider $(x,y) \in T_{R_T}$. Then, by Theorem \ref{binaria asociada a normal}, we have $(x,(y]) \in R_T$. Hence, by \eqref{T sombrero} and Remark \ref{ideal decreciente}, ${\Box_T}^{-1}(H_X(x))\subseteq H_X(y)$. It is not hard to see that, since $T$ is a meet-relation, the latter implies that $y\in T(x)$. Hence, $T=T_{R_T}$.

(2) Suppose that $x\in \Box_T(U)$. Then $T(x)\subseteq U$. To show that $x\in m_{R_T}(U)$, we need to prove that $R_T(x)\subseteq L_U$. Let $Z\in R_T(x)$. Then, by Theorem \ref{theo: MoS-spaces to Normal spaces}, $T(x)\cap Z\neq \emptyset$. Hence, there exists $y\in Z$ such that $y\in T(x)$. By hypothesis, we obtain $y\in U$. This shows that $Z\in L_U$. Therefore, $x\in m_{R_T}(U)$, as desired.

Conversely, let $x\in m_{R_T}(U)$. To prove that $x\in \Box_T(U)$, we show that $T(x)\subseteq U$. Indeed, let $y\in T(x)$. By Proposition \ref{decreciente interseccion de abiertos}, we have $T(x)\cap (y]\neq \emptyset$. Thus, by assumption, it follows that $(y]\cap U \neq \emptyset$. It is not hard to see that this implies $y\in U$. Therefore, $T(x)\subseteq U$, and consequently, $x\in \Box_T(U)$. This concludes the proof.
\end{proof}

\begin{proposition}\label{R=R_T_R}   
Let $\langle X, \mathcal{K}, R \rangle$ be a normal $mS$-space. Then $R=R_{T_R}$ and $m_{R}=\Box_{T_R}$.
\end{proposition} 

\begin{proof}
This follows directly from Theorem \ref{binaria asociada a normal} and (N1).
\end{proof}

\subsection{The isomorphism}\label{the isomorphism of categories}

Before introducing the relational structure associated with left adjoints, let us explain the role played by the construction developed so far.  Given a RelS-space $\langle X,\mathcal{K},T\rangle$, Theorem \ref{theo: MoS-spaces to Normal spaces} shows that the induced multirelation $R_T$ completely determines the modal operator $\Box_T$, in the sense that $m_{R_T}=\Box_T$. Thus, the binary meet-relation $T$ can be faithfully encoded by means of a normal multirelation. However, this construction only captures the right adjoint component of the associated adjunction. In general, the relation $R_T$ does not contain enough information to recover the left adjoint $\Box_T^{*}$ directly.

In order to represent adjunctions relationally, we therefore need an additional structure capable of encoding the action of the left adjoint at the level of subbasic closed families. The notion of $\sigma$mS-space introduced below is designed precisely for this purpose.

\begin{definition}\label{definicion sigmaSP}
A subbasic-closed monotone S-space ($\sigma$mS-space, for short) is a structure $\langle X, \mathcal{K}, G \rangle$, where $\langle X, \mathcal{K}\rangle$ is an \textit{S-space} and $G \subseteq X \times \mathcal{C}_{{\mathcal{K}}}(X)$ is a relation such that:

\begin{enumerate}
    \item $m_G(U) = \{ x \in X \colon  \exists Y \in G(x) \text{ such that } Y \subseteq U \} \in S(X)$ for all $U \in S(X)$;
    \item $G(x) = \bigcap\{ (D_U)^{c} \colon U \in S(X) \text{ and } x \in m_G(U)^{c} \}$, for all $x \in X$.
\end{enumerate}
\end{definition}

Observe that if $\langle X, \mathcal{K}, G \rangle$ is a $\sigma$mS-space, then by \eqref{conjunto D_U}, it follows that 
\[m_G(U) = \{ x \in X \colon G(x)\cap D_U\neq \emptyset\}\]
and 
\[
G(x) = \bigcap\{ (D_U)^{c} \colon U \in S(X) \text{ and } G(x)\subseteq (D_U)^{c} \}
\]

We also note the following. From Definition \ref{definicion sigmaSP}(1), given a $\sigma$mS-space $\langle X, \mathcal{K}, G \rangle$, the structure $\langle \mathbf{S}(X), m_G \rangle$ is a monotone semilattice. This justifies the use of the term ``monotone'' in the definition of $\sigma$mS-spaces.

On the other hand, let $\langle X, \mathcal{K}\rangle$ be an S-space. We define the map 
\[
\Psi_X : \mathcal{P}(\mathcal{C}_{\mathcal{K}}(X)) \to \mathcal{P}(\mathcal{Z}(X))
\]
by
\begin{equation}\label{def. psi}
\Psi_X(\mathcal{D}) = \{ Z \in \mathcal{Z}(X) : \forall Y \in \mathcal{D},\; Y \cap Z \neq \emptyset \}.
\end{equation}

Let $\langle X, \mathcal{K}, G \rangle$ be a $\sigma$mS-space. We define the relation $R_G \subseteq X \times \mathcal{Z}(X)$ by
\[
(x,Z) \in R_G \quad \text{if and only if} \quad Z \in \Psi_X(G(x)).
\]

We stress that the proof of the following result can be carried out \emph{mutatis mutandis} from Proposition 29 of \cite{CM}. In the present setting, mS-spaces correspond to $\mathcal{S}$-monotonic DS-spaces, $\sigma$mS-spaces correspond to $\mathcal{C}$-monotonic DS-spaces, closed sets are replaced by subbasic closed sets, and saturated sets by subbasic saturated subsets. We omit the proof.

\begin{proposition}\label{R_G y G_R}
Let $\langle X, \mathcal{K}, G \rangle$ be a $\sigma$mS-space. Then $\langle X, \mathcal{K}, R_G\rangle$ is a mS-space and $m_G(U) = m_{R_G}(U)$ for every $U \in S(X)$.
\end{proposition}

We recall that by Proposition 1 of \cite{P W}, if $\mathbf{A}$ is a semilattice, then there is dual isomorphism between the posets $\langle \mathrm{Fi}(\mathbf{A}), \subseteq \rangle$ and $\langle \mathcal{C}_{\mathcal{K}_\mathbf{A}}(\mathcal{X}(\mathbf{A})),\subseteq \rangle$. More precisely, such an isomorphism is provided by the map
\[
\varphi_{\mathbf{A}} : \mathrm{Fi}(\mathbf{A})\to \mathcal{C}_{\mathcal{K}_\mathbf{A}}(\mathcal{X}(\mathbf{A}))
\]
defined as
\[
\varphi_{\mathbf{A}}(F)=\{\beta_{\mathbf{A}}(a):a\in F\}=\{P\in \mathcal{X}(\mathbf{A}):F\subseteq P\}.
\]
Whose inverse
\[
\psi_{\mathbf{A}}:\mathcal{C}_{\mathcal{K}_A}(\mathcal{X}(\mathbf{A}))\to \mathrm{Fi}(\mathbf{A})
\]
is given by
\[
\psi_{\mathbf{A}}(Y)=\{a\in A:Y\subseteq \beta_{\mathbf{A}}(a)\}.
\]

Whenever no confusion arises, we simply write $\psi$ and $\varphi$ instead of $\psi_{\mathbf{A}}$ and $\varphi_{\mathbf{A}}$, respectively.
\medskip

Given two semilattices $\mathbf{A}$ and $\mathbf{B}$ and a monotone function $f \colon A \rightarrow B$. From \cite{P W} we know that we can define a relation $G_f \subseteq \mathcal{X}(\mathbf{B}) \times \mathcal{C}_{\mathcal{K}_\mathbf{A}}(\mathcal{X}(\mathbf{A}))$ by 
\begin{equation}\label{G_f}
(P, Y) \in G_f \;\Longleftrightarrow\; \psi_{\mathbf{A}}(Y) \subseteq f^{-1}[P].
\end{equation}

\begin{proposition}\label{G_m es sigmaSP}
Let $\langle \mathbf{A}, m \rangle$ be a monotone semilattice. Then $\langle \mathcal{X}(\mathbf{A}), \mathcal{K}_\mathbf{A}, G_m \rangle$ is a $\sigma$mS-space.
\end{proposition}
\begin{proof}
We start by showing that for every $a\in A$, $m_{G_m}(\beta(a))=\beta(m(a))$. On the one hand, if $P\in \beta(m(a))$, then $\uparrow a\subseteq m^{-1}[P]$. Since $\uparrow a$ is a filter of $\mathbf{A}$, and moreover $\psi(\beta(a))=\uparrow a$, we conclude that $\beta(m(a))\subseteq m_{G_m}(\beta(a))$, so $P\in m_{G_m}(\beta(a))$. Conversely, if $P\in m_{G_m}(\beta(a))$, then, there exists $Y\in \mathcal{C}_{\mathcal{K}_A}(\mathcal{X}(\mathbf{A}))$ such that $\psi(Y)\subseteq m^{-1}[P]$ and $Y\subseteq \beta(a)$. Thus, \[\uparrow a=\psi(\beta(a))\subseteq \psi(Y)\subseteq m^{-1}[P]\]
this yields $m(a)\in P$, so $P\in \beta(m(a))$. Hence, $G_m$ satisfies Definition \ref{definicion sigmaSP}(1). Now we show that for all $P \in \mathcal{X}(\mathbf{A})$:

\[
G_m(P) = \bigcap\{ (D_{\beta(a)})^{c} \colon \beta(a) \in S(\mathcal{X}(\mathbf{A}))\text{ and } G_m(P)\cap D_{\beta(a)}=\emptyset\} 
\]
To lighten the notation, put
\[
S=\{(D_{\beta(a)})^{c} : \beta(a)\in\mathcal{S}(\mathcal{X}(A)) \text{ and } G_m(P)\cap D_{\beta(a)}=\emptyset\}.
\]
We must show that $G_m(P)=\bigcap S$.
To this end, suppose that $Y\in \bigcap \mathcal{S}$. Then, for all $a\in A$ such that $G_m(P)\cap D_{\beta(a)}=\emptyset$ it follows that $Y\notin D_{\beta(a)}$. We will prove that $\psi(Y)\subseteq m^{-1}[P]$. Let $b\in \psi(Y)$. Then, since $Y\in D_{\beta(b)}$, from our assumption, $G_m(P)\cap D_{\beta(a)}$ must be non-empty. Thus, there exists $C\in \mathcal{C}_{\mathcal{K}_A}(\mathcal{X}(\mathbf{A}))$ such that $\psi(C)\subseteq m^{-1}[P]$ and $C\subseteq \beta(b)$. From this, it readily follows that $b\in m^{-1}[P]$. Consequently, $\bigcap \mathcal{S}\subseteq G_m(P)$. The remaining inclusion is straightforward. This shows that  $G_m$ satisfies Definition \ref{definicion sigmaSP}(2). Therefore, $\langle \mathcal{X}(\mathbf{A}), \mathcal{K}_\mathbf{A}, G_m \rangle$ is a $\sigma$mS-space, as claimed.
\end{proof}

The previous results show that $\sigma$mS-spaces provide the appropriate relational environment for representing monotone operators through families of subbasic closed sets.

\medskip

We now proceed to establish the isomorphism. We start with a result that will be employed in what follows.

\begin{lemma}\label{adjuncion imagen inversa}
    Let $\langle \mathbf{A},i,d \rangle$ an SLata. If $\mathbf{A}^{+}$ denotes the semilattice $\langle \mathrm{Up}(\mathbf{A}), \cap, A\rangle$, then $\langle\mathbf{A}^{+},d^{-1},i^{-1}\rangle$ is an SLata. I.e. for all $X,Y \in \mathrm{Up}(\mathbf{A})$ the following holds:
    \begin{center}
        $d^{-1}[X]\subseteq Y \iff X\subseteq i^{-1}[Y].$
    \end{center}
\end{lemma}
\begin{proof}
    It is not hard to see that both, $d^{-1}$ and $i^{-1}$ are well defined. Now, in order to show that $d^{-1}$ is the left adjoint of $i^{-1}$, let $X,Y \in \mathrm{Up}(\mathbf{A})$. Assume that  $d^{-1}[X]\subseteq Y $ and consider $a \in X$. Since $a \leq d(i(a))$, from the assumption on $X$ we obtain $d(i(a)) \in X$. Thus, $i(a)\in d^{-1}[X]$ and by hypothesis we have $i(a)\in Y$. The other implication is proved by contraposition. So assume that $X\subseteq i^{-1}[Y]$. If $y\notin Y$, then because $i(d(y))\leq y$ and the assumption on $Y$, we get $i(d(y))\notin Y$. So, $d(y)\notin X$ by assumption. Consequently, $y\notin d^{-1}[X]$. 
\end{proof}

\begin{definition}\label{G_T}
    Let $\langle X, \mathcal{K}, T \rangle$ be a MoS-space. We define the relation $G_T\subseteq X \times \mathcal{C_K}(X)$ by
    \begin{center}
       $ (x,Y) \in G_T \iff x\in \bigcap\{U\in S(X)\colon Y\in D_{\Box_T(U)}\}.$
    \end{center}
\end{definition}

\begin{proposition}\label{G_T sii G_box T*}
Let $\langle X, \mathcal{K}, T \rangle$ be a RelS-space. Then, for every $x\in X$ and $Y\in \mathcal{C}_{\mathcal{K}}(X)$, we have:
\[
(x,Y) \in G_T \iff (H_X(x), H_X[Y]) \in G_{\Box_T^{*}}.
\]
Thus, $G_{\Box^{\ast}_T}(H_X(x))=H_X[G_T(x)]$, for every $x\in X$.
\end{proposition}
\begin{proof}
   We only show the left-to-right implication, since the proof of the converse implication is analogous. Let $x\in X$ and $Y \in \mathcal{C_K}(X)$. If $(x,Y) \in G_T$, then, for every $U\in S(X)$, if $Y\subseteq \Box_T(U)$, then $x \in U$. To prove our claim, we need to show that $\psi_{\mathbf{S}(X)}(H_X[Y]) \subseteq{\Box_T^{*}}^{-1}[H_X(x)]$. Observe that, by Lemma \ref{adjuncion imagen inversa}, the latter is equivalent to ${\Box_T}^{-1}[\psi_{\mathbf{S}(X)}(H_X[Y])] \subseteq H_X(x)$. Thus, we prove this inclusion. Let $U\in S(X)$ and suppose that ${\Box_T}(U)\in \psi_{\mathbf{S}(X)}(H_X[Y])$. Then, $H_X[Y]\subseteq \beta_{\mathbf{S}(X)}({\Box_T}(U))$. By Lemma \ref{lem: cosas de H_X}(1) and the properties of $\beta_{\mathbf{S}(X)}$, we obtain $Y\subseteq \Box_T(U)$. Hence, by assumption, it follows that $x\in U$. Consequently, $U\in H_X(x)$. This concludes the proof.
\end{proof}

\begin{proposition}\label{prop: Rel-Sp to sigmaSp}
    Let  $\langle X, \mathcal{K}, T \rangle$ be a RelS-space. Then $\langle X, \mathcal{K}, G_T \rangle$ is a $\sigma$mS-space.
\end{proposition}
\begin{proof}
Let $U \in S(X)$. Since $\langle \mathbf{S}(X),\Box_T^{*}\rangle$ is a monotone semilattice, Proposition \ref{G_m es sigmaSP} yields that $\langle \mathcal{X}(\mathbf{S}(X)), \mathcal{K}_{\mathbf{S}(X)}, G_{\Box_T^{*}}\rangle$ is a $\sigma$mS-space. Hence, by the last part of Proposition \ref{G_T sii G_box T*}, condition (1) of Definition \ref{definicion sigmaSP} follows immediately. Now we prove condition (2) of Definition \ref{definicion sigmaSP}. Let $U \in S(X)$. By Lemma \ref{lem: cosas de H_X}(1), we have that $H_X[D_U]=D_{H_X[U]}$. Together with Proposition \ref{G_T sii G_box T*}, this yields
\[
m_{G_{\Box_T^{*}}}(\beta_{\mathbf{S}(X)}(U))=H_{X}[m_{G_T}(U)].
\]
Since $\langle \mathcal{X}(\mathbf{S}(X)), \mathcal{K}_{\mathbf{S}(X)}, G_{\Box_T^{*}}\rangle$ is a $\sigma$mS-space, the previous equality and the bijectivity of $H_X$ allow us to conclude that condition (2) of Definition \ref{definicion sigmaSP} holds. Therefore, $\langle X, \mathcal{K}, G_T \rangle$ is a $\sigma$mS-space, as desired.
\end{proof}

\begin{lemma}\label{S_G_T=T^{-1}}
    Let $\langle X, \mathcal{K}, T\rangle$ an RelS-space. Then, for all $x,y\in X$,
    \[(x,y)\in T^{-1} \iff (x,\uparrow y) \in G_T. \]
\end{lemma}
\begin{proof}
    We start by noticing that by Proposition \ref{def box T^{*}}, $\langle \mathbf{S}(X), \Box_T^{\ast}, \Box_T\rangle$ is an SLata. So, in particular, $\langle \mathbf{S}(X),\Box_T\rangle$ is a modal semilattice. Suppose that $(x,y)\in T^{-1}$, then, by Theorem \ref{dualidad RS espaces}, it follows that $(H_X(y), H_X(x))\in N_{\Box_T}$. Thus, $\Box_T^{-1}[H_X(y)]\subseteq H_X(x)$ and by Lemma \ref{adjuncion imagen inversa}, we get $H_X(y)\subseteq {\Box_T^{\ast}}^{-1}[H_X(x)]$. Since $\psi_{\mathbf{S}(X)}[\uparrow H_X(y)]=H_X(y)$, then, the latter, \eqref{G_f} and Lemma \ref{lem: cosas de H_X}(4) implies that $(H_X(x),H_X(y))\in G_{{\Box_T^{\ast}}^{-1}}$. Hence, $(x, \uparrow y)\in G_T$. The proof of the converse is analogous.
\end{proof}

We now show that the structure induced by $G_T$ is sufficient to recover the left adjoint associated with the meet relation $T$.

\begin{theorem}\label{de Rel a SLataSp}
   Let $\langle X, \mathcal{K}, T\rangle$ be a RelS-space. Then $\langle X, \mathcal{K}, R_{G_T}, R_T\rangle$ is a SLata-space.
\end{theorem}
\begin{proof}
First, notice that by Theorem \ref{theo: MoS-spaces to Normal spaces}, $\langle X, \mathcal{K},R_T\rangle$ is a mS-space. Furthermore, propositions \ref{prop: Rel-Sp to sigmaSp} and \ref{R_G y G_R}, imply that $\langle X, \mathcal{K}, R_{G_T}\rangle$ is a mS-space and $m_{R_{G_T}}=m_{G_T}$. By Lemma \ref{T=T_R_T} it also holds that $m_{R_T}=\Box_T$. So, in order to prove our claim, by Proposition 3.1 of \cite{GPZ}, we only need to prove that $\langle \mathbf{S}(X),m_{R_{G_T}}, m_{R_T}\rangle$ is a SLata. 
 
To that end, we will prove that $m_{G_T}(\Box_T(U))\subseteq U$ and $U \subseteq \Box_T(m_{G_T}(U))$, for every $U \in S(X)$. Let $U\in S(X)$. On the one hand, if $x \in m_{G_T}(\Box_T(U))$, from Definition \ref{definicion sigmaSP} we have that there exists $Y \in G_T(x)$ such that $Y \subseteq \Box_T(U)$. Thus, from the above and by Definition \ref{G_T} we have that $x \in U$. For the converse, suppose that $x \in U$. In order to prove our claim, we need to show that $T(x)\subseteq m_{G_T}(U)$. Indeed, if $y \in T(x)$, then, by Lemma \ref{S_G_T=T^{-1}} we have that $(y, \uparrow x) \in G_T$. By hypothesis we know that $\uparrow x\subseteq U$, so since $\uparrow x=\mathrm{cl}_{\mathcal{K}}(x)$, from Definition \ref{definicion sigmaSP} we obtain that $y \in m_{G_T}(U)$. Hence, $\langle \mathbf{S}(X),m_{R_{G_T}}, m_{R_T}\rangle$ is a SLata and consequently, $\langle X, \mathcal{K}, R_{G_T}, R_T\rangle$ is a SLata-space.
\end{proof}

The following two results show that, for a meet-relation 
$M\subseteq X_1\times X_2$ between S-spaces $\langle X_1, \mathcal{K}_1 \rangle$ 
and $\langle X_2, \mathcal{K}_2\rangle$, being an SLata-relation (i.e., 
monotone with respect to both multirelational components) is equivalent 
to being adjoint-preserving in the sense of Definition~\ref{relacion apropiada}. 
In this sense, both notions of morphism are two sides of the same coin, 
and these results are key to establishing the main theorem of this section.

\begin{lemma}\label{de Rel a SLataSp: morfismos}
   Let $M\subseteq X_1\times X_2$ be a meet-relation between the S-spaces $\langle X_1, \mathcal{K}_1 \rangle$ and $\langle X_2, \mathcal{K}_2\rangle$. Then, the following are equivalent: 
   \begin{enumerate}
       \item $M$ is an adjoint-preserving relation between the RelS-spaces $\langle X_1, \mathcal{K}_1, T_1\rangle$ and $\langle X_2, \mathcal{K}_2, T_2\rangle$. 
       \item $M$ is an SLata relation between the SLata-spaces $\langle X_1, \mathcal{K}_1, R_{G_{T_1}}, R_{T_1}\rangle$ and  $\langle X_2, \mathcal{K}_2, R_{G_{T_2}}, R_{T_2}\rangle$.
   \end{enumerate}
\end{lemma}
\begin{proof}
Let us assume that $\langle X_i, \mathcal{K}_i, T_i\rangle$ and 
$\langle X_i, \mathcal{K}_i, R_{G_{T_i}}, R_{T_i}\rangle$, with $i=1,2$, 
are RelS-spaces and SLata-spaces, respectively. Since each $T_i$ is an 
A-relation, by Definition~\ref{A relacion}(1) it is in particular a 
meet-relation, so $\langle X_i, \mathcal{K}_i, T_i\rangle$ is a MoS-space 
for $i=1,2$, and consequently, by Lemma~\ref{T=T_R_T}(2), 
$m_{R_{T_i}} = \Box_{T_i}$. By Theorem~\ref{duality SLata}, 
$\langle \mathbf{S}(X_i), m_{R_{G_{T_i}}}, m_{R_{T_i}}\rangle$ is an SLata, 
so in particular $m_{R_{G_{T_i}}} \dashv m_{R_{T_i}}$. Furthermore, 
Proposition~\ref{def box T^{*}} yields $\Box^*_{T_i} \dashv \Box_{T_i}$, 
hence since left adjoints are unique, $m_{R_{G_{T_i}}} = \Box^*_{T_i}$. 
Therefore, the SLata $\langle \mathbf{S}(X_i), m_{R_{G_{T_i}}}, m_{R_{T_i}}\rangle$ 
coincides with $\langle \mathbf{S}(X_i), \Box^*_{T_i}, \Box_{T_i}\rangle$. 
By Theorem~\ref{duality SLata}, $M$ satisfies (2) if and only if $\Box_M$ 
is a SLata homomorphism between 
$\langle \mathbf{S}(X_2), \Box^*_{T_2}, \Box_{T_2}\rangle$ and 
$\langle \mathbf{S}(X_1), \Box^*_{T_1}, \Box_{T_1}\rangle$, 
which by Lemma~\ref{funtor morfirmo M} holds if and only if $M$ 
satisfies (1).
\end{proof}

\begin{lemma}\label{de SLataSp a Rel: morfismos}
   Let $M\subseteq X_1\times X_2$ be a meet-relation between the S-spaces $\langle X_1, \mathcal{K}_1 \rangle$ and $\langle X_2, \mathcal{K}_2\rangle$. Then, the following are equivalent: 
   \begin{enumerate}
       \item $M$ is an SLata relation between the SLata-spaces $\langle X_1, \mathcal{K}_1, I_1, E_1\rangle$ and $\langle X_2, \mathcal{K}_2, I_2, E_2\rangle$.
       \item $M$ is an adjoint-preserving relation between the RelS-spaces $\langle X_1, \mathcal{K}_1, T_{E_1}\rangle$ and $\langle X_2, \mathcal{K}_2, T_{E_2}\rangle$.       
   \end{enumerate}
\end{lemma}
\begin{proof}
Since $\langle X_i, \mathcal{K}_i, I_i, E_i\rangle$ is a SLata-space, 
by Theorem \ref{duality SLata} the structure $\langle \mathbf{S}(X_i), 
m_{I_i}, m_{E_i}\rangle$ is an SLata, so $m_{E_i}$ is a modal operator. 
Hence, by Lemma~\ref{dual normal}, the  space 
$\langle \mathcal{X}(\mathbf{S}(X_i)), \mathcal{K}_{\mathbf{S}(X_i)}, 
R_{m_{E_i}}\rangle$ is a normal $mS$-space. By Proposition~3.4 of 
\cite{GPZ}, the map $H_{X_i}$ induces an isomorphism of $mS$-spaces 
between $\langle X_i, \mathcal{K}_i, E_i\rangle$ and 
$\langle \mathcal{X}(\mathbf{S}(X_i)), \mathcal{K}_{\mathbf{S}(X_i)}, 
R_{m_{E_i}}\rangle$. It is not hard to see that normality is preserved under isomorphisms 
of $mS$-spaces, thus $\langle X_i, \mathcal{K}_i, E_i\rangle$ is a normal 
$mS$-space. Therefore, by Proposition~\ref{R=R_T_R}, $E_i = R_{T_{E_i}}$, and by 
Theorem~\ref{binaria asociada a normal}, $\langle X_i, \mathcal{K}_i, 
T_{E_i}\rangle$ is a MoS-space. By Lemma~\ref{T=T_R_T}(2),
\[
\Box_{T_{E_i}} = m_{R_{T_{E_i}}} = m_{E_i}.
\]

By Theorem~\ref{Relational duality for SLatas}, 
$\langle \mathbf{S}(X_i), \Box^*_{T_{E_i}}, \Box_{T_{E_i}}\rangle$ 
is an SLata. Since $m_{E_i} = \Box_{T_{E_i}}$ and left adjoints are 
unique, $m_{I_i} = \Box^*_{T_{E_i}}$, so both SLatas coincide:
\[
\langle \mathbf{S}(X_i), m_{I_i}, m_{E_i}\rangle 
= 
\langle \mathbf{S}(X_i), \Box^*_{T_{E_i}}, \Box_{T_{E_i}}\rangle.
\]

Therefore, by Theorem~\ref{duality SLata} and 
Lemma~\ref{funtor morfirmo M}, $M$ satisfies (1) if and only if 
$M$ satisfies (2).
\end{proof}

Based on Theorem \ref{de Rel a SLataSp} and Lemma \ref{de Rel a SLataSp: morfismos}, it is clear that the assignments
\[
\begin{array}{rcl}
   \langle X, \mathcal{K}, T\rangle & \mapsto & \langle X, \mathcal{K}, R_{G_T}, R_T\rangle, \\
   M & \mapsto & M, 
\end{array}
\]
define a functor \(\mathbf{P} : \mathsf{RelSp} \to \mathsf{SLataSp}\).
\medskip

Conversely, as shown in the proof of Lemma~\ref{de SLataSp a Rel: morfismos}, if 
$\langle X, \mathcal{K}, I, E\rangle$ is a SLata-space then 
$\langle X, \mathcal{K}, E\rangle$ is a normal $mS$-space, 
so $T_E$ is well defined. Thus, the assignments
\[
\begin{array}{rcl}
    \langle X, \mathcal{K}, I,E\rangle  & \mapsto & \langle X, \mathcal{K},
T_E\rangle, \\
     M & \mapsto & M, 
\end{array}
\]
define a functor $\mathbf{Q}:\mathsf{SLataSp} \rightarrow \mathsf{RelSp}$.
\medskip

Now we prove the last result of the paper.

\begin{theorem}\label{last theorem}
    The categories $\mathsf{SLataSp}$ and $\mathsf{RelSp}$ are isomorphic.
\end{theorem}
\begin{proof}
We will prove the following: (a) $\mathbf{Q}\mathbf{P}$ coincides with the identity functor on $\mathsf{RelSp}$; and (b) $\mathbf{P}\mathbf{Q}$ coincides with the identity functor on $\mathsf{SLataSp}$. By 
Lemmas~\ref{de Rel a SLataSp: morfismos} and \ref{de SLataSp a Rel: morfismos}, 
both $\mathbf{P}$ and $\mathbf{Q}$ are the identity on morphisms, so it 
suffices to verify these equalities on objects. Observe that (a) follows from Lemma \ref{T=T_R_T}(1). On the other hand, let $\langle X, \mathcal{K}, I,E\rangle$ be a SLata-space. Then we get that 
\[\mathbf{P}\mathbf{Q}(\langle X, \mathcal{K}, I,E\rangle)=\langle X, \mathcal{K}, R_{G_{T_E}},R_{T_E}\rangle.\]
By Proposition~\ref{R=R_T_R}, $E = R_{T_E}$ and consequently 
$m_E = m_{R_{T_E}}$. Since $\langle \mathbf{S}(X), m_I, m_E\rangle$ 
is an SLata, the uniqueness of the left adjoint yields 
$m_I = m_{R_{G_{T_E}}}$. Since both $\langle X, \mathcal{K}, I\rangle$ and 
$\langle X, \mathcal{K}, R_{G_{T_E}}\rangle$ are $mS$-spaces, 
condition~(m2) of Definition~\ref{def: mS-space} yields $I = R_{G_{T_E}}$. 
Hence, $\mathbf{P}$ and $\mathbf{Q}$ establish the desired isomorphism.
\end{proof}

\begin{remark}
Theorem \ref{last theorem} upgrades an equivalence that was already available at the abstract level to an isomorphism. Indeed, $\mathsf{SLata}^{op}$ is equivalent to $\mathsf{RelSp}$ via the functors $\mathbf{M}$ and $\mathbf{R}$ (Theorem~\ref{Relational duality for SLatas}), and $\mathsf{SLata}^{op}$ is equivalent to $\mathsf{SLataSp}$ by Theorem~\ref{duality SLata}, so the equivalence between $\mathsf{SLataSp}$ and $\mathsf{RelSp}$ follows by composing these two equivalences. However, this argument does not identify the underlying spaces and relations. The functors $\mathbf{P}$ and $\mathbf{Q}$ do: they are mutually inverse and the identity on morphisms, and correspond to $\langle X, \mathcal{K}, I, E\rangle \leftrightarrow \langle X, \mathcal{K}, T_E\rangle$.
\end{remark}

\section{Acknowledgements} 
Both authors want to express their gratitude to the referees for the careful reading of this manuscript. Their suggestions and comments allowed us to strongly improve the presentation of this paper. This research has been 
funded by the CONICET [PIP 11220250100318CO] and the European Research
Council (ERC) under the European Union's Horizon Europe research and 
innovation programme, Marie Skłodowska-Curie Actions, grant agreement 
No.~101299559 (SemPER) \url{https://cordis.europa.eu/project/id/101299559}.


-----------------------------------------------------------------------------------
\\
Belén Gimenez, \orcidlink{0009-0009-0427-0861}\\
Departamento de Matem\'atica,\\
Facultad de Ciencias Exactas (UNLP),\\
50 y 115, La Plata (1900)\\
belengim.28@gmail.com

-----------------------------------------------------------------------------------
\\
William Zuluaga \orcidlink{0000-0002-8798-9493},\\
Facultad de Ciencias Exactas (UNCPBA),\\
Pinto 399, Tandil (7000),\\
and CONICET, Argentina,\\
wizubo@gmail.com

\end{document}